\documentclass{amsart}

\makeatletter
 \def\LaTeX{\leavevmode L\raise.42ex
   \hbox{\kern-.3em\size{\sf@size}{0pt}\selectfont A}\kern-.15em\TeX}
\makeatother

\newcommand{\BibTeX}{{\rm B\kern-.05em{\sc
i\kern-.025emb}\kern-.08em\TeX}}

\newtheorem{theorem}{Theorem}[section]

\newtheorem{definition}{Definition}[section]

\newtheorem{remark}{Remark}

\numberwithin{equation}{section}

\begin{document}

\title{Sampling, splines and frames on compact manifolds}

\maketitle
\begin{center}

\author{Isaac Z. Pesenson }\footnote{ Department of Mathematics, Temple University,
 Philadelphia,
PA 19122; pesenson@temple.edu}

\end{center}

\tableofcontents

\begin{abstract}

Analysis on the unit sphere $\mathbb{S}^{2}$ found
many applications in  seismology, weather prediction, astrophysics, signal analysis, crystallography,  computer vision, computerized tomography, neuroscience, and statistics. 
	In the last two decades, the importance of  these and other applications triggered the development of various tools such as splines and wavelet bases suitable for the unit spheres $\mathbb{S}^{2}$,  $\>\>\mathbb{S}^{3}$ and  the rotation group  $SO(3)$. Present paper is  a summary of some of results of the author and his collaborators on the Shannon-type sampling, generalized (average) variational splines and localized frames (wavelets) on compact Riemannian manifolds. The results are illustrated by applications to Radon-type transforms on $\mathbb{S}^{d}$ and $SO(3)$.

\end{abstract}
 
\section{Introduction}

Harmonic analysis on the unit  spheres $\mathbb{S}^{2}$ and  $\mathbb{S}^{3}$, on the rotation group of $\mathbb{R}^{3}$ and even on more general manifolds 
found
many applications in  seismology, weather prediction, astrophysics, signal analysis,  computer vision, computerized tomography, neuroscience, scattering theory and statistics. 
Our list of references in which different aspects of analysis on manifolds were either developed or applied 
is very far from being complete \cite{ANS}-\cite{BPS}, \cite{CM}-\cite{DH}, \cite{DFHMP}-\cite{GM100}, \cite{HMS}-\cite{SAZ}. More references can be found in monographs \cite{AH}, \cite{FGS98}, \cite{FM04}, \cite{LS1}, \cite{MP}.

Cubature formulas on spheres and interpolation on spheres can be traced back to the classical papers by S. L. Sobolev \cite{Sob} and I. J. Schoenberg \cite{S}.
Pioneering work on splines, interpolation and approximation on spheres with many applications to geophysics was done by W. Freeden and his school  \cite{F78}-\cite{FS07}.  Spherical splines for statistical analysis were developed  in \cite{W1}-\cite{W2}. Important construction of the so-called needlets on $\mathbb{S}^{n}$ appeared in \cite{NPW}.

	The goal of the present study is to describe new constructions and applications  of generalized splines and Parseval
	 bandlimited and localized frames in a space $L_{2}(M)$, where $M$ is a compact homogeneous Riemannian manifold. 
  Our article is a summary of some of  results  that were obtained by author and his collaborators in   \cite{BEP}, \cite{BP},  \cite{DFHMP}, \cite{FP04}, \cite{gpes-1},  \cite{gpes-2},  \cite{KP}, \cite{Pes88b}-\cite{Pes13}. To the best of our knowledge these are the papers which contain the most general and comprehensive  results about splines, frames and Shannon sampling on compact Riemannian manifolds along with applications to Radon-type transforms on manifolds.

The paper is organized as follows. In section \ref{Sh-sampling} we develop what can be called the Shannon-type sampling of bandlimited functions on compact Riemannian manifolds. It is shown that proposed rate  of sampling  of bandlimited functions is essentially optimal. By bandlimited functions on compact manifolds we understand polynomials in eigenfunctions of the corresponding Laplace-Beltrami operator. 

 We also introduce variational splines which are used to interpolate smooth functions on finite sets of points. It should be noted that splines exhibit a strong localization compare to other natural "bases" on manifolds (compare, say, to eigenfunctions of the corresponding Laplace-Beltrami operators). It makes them (like in the classical situation) an indispensable  tool when it is necessary to reduce influence of a noise inherited into point-wise measurements.

We discuss a method of reconstruction of bandlimited functions from sampling sets as limits of variational splines when smoothness of splines goes to infinity. 
 Another method of reconstruction by using iterations is also discussed.  These results are based on our extension of the Plancherel-Polya inequalities to manifolds (Theorem \ref{PP}). It should be noted that in the case of manifolds (compact and non-compact) such inequalities appeared first in our papers \cite{Pes00} and \cite{Pes04b}. Later on it became common to call them Marcinkiewicz-Zygmund inequalities when one refers to compact manifolds.

In subsection \ref{splines-1} we introduce what we call \textit{generalized} variational interpolating splines on compact Riemannian manifolds \cite{Pes00}-\cite{Pes07}.
We are motivated by the following problem which is of
interest for integral geometry.
Let $M, \dim M=n,$ be a Riemannian manifold and
 $\mathcal{M}_{\nu}, \nu=1,2,...,N,$
 is a family of submanifolds
 $\dim \mathcal{M}_{\nu}=d_{\nu}, 0\leq d_{\nu}\leq n.$ Given a
set of numbers $v_{1},v_{2},...,v_{N}$ we would like to find a
function for which
\begin{equation}\label{samples}
\int_{\mathcal{M}_{\nu}}fdx=v_{\nu}, \>\>\>\nu=1,2,...,N.
\end{equation}

Moreover, we are interested in a "least curved" function that
satisfies the previous conditions. In other words, we seek a
function that satisfies (\ref{samples}) and minimizes the functional
$$
u\rightarrow \|(1-L)^{t/2}u\|_{L_{2}(M)},\>\>\>t\in \mathbb{R},
$$ 
where $-L$ is a differential second order elliptic operator which is self-adjoint and non-negative in the natural space $L_{2}(M)$.
Note that
in the case when the submanifold $\mathcal{M}_{\nu}$ is a point the integral
(\ref{samples}) is understood as a value of a function at this point.

Our result is that if $s$ is a solution of such variational
problem then the distribution $(1-L)^{t}s$ should satisfy the
following distributional pseudo-differential equation on $M$ for
any $\psi\in C_{0}^{\infty}(M)$,
\begin{equation}
\int_{M}\psi  (1-L)^{t}s
dx=\sum_{\nu=1}^{N}\alpha_{\nu}\int_{\mathcal{M}_{\nu}}\psi dx,
\end{equation}
where coefficients $\alpha_{\nu}\in \mathbb{C}$ depend just on
$s$.
This equation allows one to obtain the Fourier coefficients of the
function $s$ with respect to eigenfunctions of the
operator $L$.
From the very definition our solution $s$ is an "interpolant" in
the sense that it has a prescribed set of integrals. Moreover, we
show that the function $s$ is not just an interpolant but also an
optimal approximation to the set of all functions $f$ in the
Sobolev space $H_{t}(M)$ that satisfy (\ref{samples}) and
\begin{equation}\label{derivative}
\|(1-L)^{t/2}f\|_{L_{2}(M)}\leq K,
\end{equation}
for appropriate $K >0.$ Namely, we show that $s$ is the center of
the convex and bounded set of all functions that satisfy (\ref{samples}) and
(\ref{derivative}).

We use generalized splines for approximate inversion of the Funk-Radon transform, hemispherical transform, and the Radon transform on the group $SO(3)$.  The corresponding integral transforms are  introduced in section \ref{splines}.

In  subsections \ref{Radon-spheres} a brief introduction to the Funk-Radon transform $R$ is given. 
In this case the manifold 
 $M$ is the standard unit sphere $\mathbb{S}^{n}$ in $\mathbb{R}^{n+1}$, every $\mathcal{M}_{\nu},\>\>\nu=1,...,N,$ is a great subsphere of $ \mathbb{S}^{n}$ and $L$ is the Laplace-Beltrami operator on $\mathbb{S}^{n}$.  The Radon transform $R$  transforms even functions   on $\mathbb{S}^{n}$ into even functions and it is invertible. 
 
Our objective in subsection \ref{inverF-R}  is to find an approximate preimage  of $Rf$ where $f$ is sufficiently smooth and even by using only a finite  set of  values $\{v_{\nu}\}_{1}^{N}$ defined in (\ref{samples}). We achieve the goal (Theorem \ref{Main-Applic}) by constructing an even smooth function $s_{t}(f)$ ($t$ is a smoothness parameter)  which interpolates $f$ in the sense that it has the same set of integrals over subspheres $\{w_{\nu}\}_{1}^{N}$. This interpolant is an optimal approximation to $f$ in the sense explained above.

In subsection \ref{pointwise} we explore a different approach which  allows to obtain a quantitative estimate on the rate of convergence of our approximations. We consider a set of points $\{x_{\nu}\}_{1}^{N}$ which is dual to the set of subspheres $\{\mathcal{M}_{\nu}\}_{1}^{N}$.  Let $\rho>0$ be a separation parameter for the mesh  $\{x_{\nu}\}_{1}^{N}$. The first step now is to use values $\{Rf(x_{\nu})\}_{1}^{N}$ to construct a spline $s_{\rho, \tau}(Rf)$ ($\tau$ is a smoothness parameter) which interpolates and approximates the Radon transform $Rf$.  For the even function $S_{\rho, \tau}(f)=R^{-1}s_{\rho, \tau}(Rf)$ we are able to show (see Theorem \ref{last-thm555}) that a difference  between   $S_{\rho, \tau}(f)$ and  $f$ is of order $\rho$ to a power.  In other words when $\rho$ goes to zero a sequence of corresponding functions $S_{\rho, \tau}(f)$ converges to $f$ with a geometrical rate. Moreover, the Theorem \ref{SSTh} shows that if $f$ is $\omega$-bandlimited and $\rho$ is small enough compare to $\omega$ than $S_{\rho, \tau}(f)$ converges to $f$ when smoothness $\tau$ goes to infinity (for a sufficiently dense but \textit{fixed} mesh $\{x_{\nu}\}_{1}^{N}$). This result can be considered as a generalization of the Classical Sampling Theorem.

Hemispherical Radon transform on $\mathbb{S}^{n}$ is introduced in subsection \ref{hem}.
In subsections \ref{inverHem} and \ref{samplHem}  we realize for the hemispherical Radon transform the same  program which was realized in  \ref{inverF-R} and   \ref{pointwise} for the Funk-Radon transform.

In \ref{RadSO} we briefly describe some basic properties of the Radon transform $\mathcal{R}$ on the group of rotations $SO(3)$. Approximate inversion of the Radon transform on $SO(3)$ using  generalized splines is considered in \ref{approxSO}. A sampling theorem for the Radon transform $\mathcal{R}$ of bandlimited functions on $SO(3)$ is proved in \ref{SamplRSO}.

In subsections \ref{homman} and \ref{frames} we construct bandlimited and localized  Parseval frames in $L_{2}(M)$ where $M$ is a compact homogeneous manifold \cite{gpes-1}, \cite{pg}. 

 Let us remind \cite{Gr} that a  set of vectors $\{\theta_{v}\}$  in a Hilbert space $\mathcal{H}$ is called a Hilbert  frame if there exist constants $A, B>0$ such that for all $f\in \mathcal{H}$ 
\begin{equation}\label{Frame ineq}
A\|f\|_{L_{2}(M)}^{2}\leq \sum_{v}\left|\left<f,\theta_{v}\right>\right|^{2}     \leq B\|f\|_{L_{2}(M)}^{2}.
\end{equation}
The largest $A$ and smallest $B$ are called respectively the lower and the upper frame bounds and the ratio $B/A$ is known as the tightness of the frame.  If $A=B$ then $\{\theta_{v}\}$ is a \textit{tight } frame, and if $A=B=1$ it is called a \textit{Parseval } frame. 
Parseval frames are similar in many respects to orthonormal bases.  For example, if all members of a Parseval frame are unit vectors then it  is an  orthonormal basis. 

According to the general theory of Hilbert frames \cite{DS}, \cite{Gr} the frame  inequality (\ref{Frame ineq})  implies that there exists a dual frame $\{\Theta_{v}\}$ (which is not unique in general) 
for which the following reconstruction formula holds 
\begin{equation}
f=\sum_{v}\left<f,\theta_{v}\right>\Theta_{v}.
\end{equation}
The important fact is that in the case of a Parseval frame one can take  $\Theta_{v}=\theta_{v}$.

Using our Theorems \ref{cubature} and \ref{prodthm} we construct a Parseval frame  in $L_{2}(M)$ (where $M$ is compact and homogeneous) whose elements are bandlimited and have very strong localization on $M$. Such frame is a substitute for a wavelet basis and it is perfectly suitable to perform multiresolution  analysis on compact homogeneous manifolds.

It is worth to remind that from the point of view of computations bandlimited localized frames on manifolds have a number of advantages compare, say, to orthonormal bases  comprised of eigenfunctions of respected Laplace-Beltrami operators. The main of these advantages are  localization in  space and frequency and  redundancy. 

In subsections \ref{exactFCS}  and \ref{Discrete} we are using our localized Parseval frames on  $\mathbb{S}^{n}$ and on $SO(3)$ to develop  exact inversion of the  Funk-Radon and Radon transforms  of bandlimited functions on these manifolds from   finite sets of samples.

\section{ Shannon-type sampling on compact Riemannian manifolds }\label{Sh-sampling}

\subsection{Compact Riemannian manifolds}

Let $M$ be a compact Riemannian manifold without boundary and   $L$ be a differential second order elliptic operator which is  self-adjoint and negatively semi-definite  in the
space $L_{2}(M)$ constructed using a Riemannian density $dx$. 
The best known example of such operator is the Laplace-Beltrami which is given in a local coordinate system  by the formula
$$
\Delta
f=\sum_{m,k}\frac{1}{\sqrt{det(g_{ij})}}\partial_{m}\left(\sqrt{det(g_{ij})}
g^{mk}\partial_{k}f\right)
$$
where $g_{ij}$ are components of the metric tensor,$\>\>det(g_{ij})$ is the determinant of the matrix $(g_{ij})$, $\>\>g^{mk}$ components of the matrix inverse to $(g_{ij})$.

 In order to have an invertible operator we will work with $I-L$, where $I$ is the identity operator in $L_{2}(M)$.  It is known that for every such operator $L$ the domain of the power $(I-L)^{t/2},\>t\in \mathbb{R}$, is the Sobolev space $H_{t}(M),\>\>t\in \mathbb{R}$. There are different ways to introduce norm in Sobolev spaces. For a fixed operator $L$ we will introduce the graph norm as follows. 
\begin{definition}
\label{Sobolevnorm}
The Sobolev space $H_{t}(M), t\in \mathbb{R}$ can be introduced as the
domain of the operator $(I-L)^{t/2}$ with the graph norm

$$
\|f\|_{ H_{t}(M)}=\|(I-L)^{t/2}f\|_{L_{2}(M)}, f\in H_{t}(M).
$$
\end{definition}

Note, that such norm depends on $L$. However, for differential secod-order elliptic operators  such norms are equivalent for each $t\in \mathbb{R}$. 

Since the operator $(-L)$ is  self-adjoint and positive semi-definite it has a discrete
spectrum $0\leq\lambda_{0}\leq\lambda_{1}\leq \lambda_{2}\leq...,$ and
one can choose corresponding eigenfunctions $u_{0},
u_{1},...$ which form an orthonormal basis of $L_{2}(M).$ A
distribution $f$ belongs to $H_{t}(M), \>\>\>t\in \mathbb{R},$ if and
only if
$$
\|f\|_{ H_{t}(M)}=
\left(\sum_{j=0}^{\infty}(1+\lambda_{j})^{t}|c_{j}(f)|^{2}\right)^{1/2}<\infty,
$$
where Fourier coefficients $c_{j}(f)$ of $f$ are given by
$$
c_{j}(f)=\left<f,u_{j}\right>=\int_{M}f(x)\overline{u_{j}(x)}dx.
$$

\begin{definition}\label{bandlimited}
The span of all eigenfunctions of $-L$ whose corresponding eigenvalues are not greater than a positive $\omega$ is denoted by ${\bf E}_{\omega}(L)$ and called the set of $\omega$-bandlimited functions.
\end{definition}

\subsection{A sampling theorem on manifolds}

One can show  \cite{Pes00}, \cite{Pes04b} that for a compact Riemannian manifold $M$ there exists a constant $N_{M}$ such that 
for  every sufficiently small $\rho>0$ there exists  a finite set of points $M_{\rho}=\{x_{\nu}\}$ such that
\begin{enumerate}

\item The balls $B(x_{\nu}, \rho/4)$ are disjoint.

\item The balls $B(x_{\nu}, \rho/2)$ form a cover of $M$.

\item The multiplicity of the cover $\{B(x_{\nu}, \rho)\}$ is not greater $N_{M}$.
\end{enumerate}

\begin{definition}
A set of points $M_{\rho}=\{x_{\nu}\}$ that satisfies (1)-(3)  will be called a metric $\rho$-lattice. 
\end{definition}

\begin{theorem}\label{PP}

There exist constants $\>\>\>A=A(M, L)>0,\>\>$
$B=B(M, L)>0,$
 and
$\>\>c_{0}=c_{0}(M, L)>0,$ such that for any $\omega>0$, and for
every metric $\rho$-lattice $M_{\rho}=\{x_{\nu}\}$ with $\rho=
c_{0}\omega^{-1/2}$, the following Plancherel-Polya inequalities
hold:

\begin{equation}
A\|f\|_{L_{p}(M)}\leq
\rho^{n/p}\left(\sum_{\nu}|f(x_{\nu})|^{p}\right)^{1/p}
\leq 
B\|f\|_{L_{p}(M)},\>\>\>1\leq p\leq \infty,  \label{completePlPo100}
\end{equation}
for all $f\in {\mathbf E}_{\omega}(L)$ and $n=\dim \  M$.
\label{completePlPo200}
\end{theorem}

\begin{definition}
For a given space ${\mathbf E}_{\omega}(L)$ every $M_{\rho}$ metric $\rho$-lattice which is described in the previous theorem will be called a sampling set for ${\mathbf E}_{\omega}(L)$.
\end{definition}

This result can be reformulated as a generalization of the Classical Sampling Theorem for bandlimited functions (see \cite{Pes00}-\cite{Pes07}).

\begin{theorem}

There exists a  constant 
$\>\>c_{0}=c_{0}(M, L)>0,$ such that for any $\omega>0$, and for
every metric $\rho$-lattice $M_{\rho}=\{x_{\nu}\}$ with $\rho=
c_{0}\omega^{-1/2}$, all functions in ${\mathbf E}_{\omega}(L)$ are uniquely determined by their values on $M_{\rho}=\{x_{\nu}\}$ and can be reconstructed from this set of values in a stable way.

\end{theorem}

The following Theorem implies that our metric lattices $M_{\rho}=\{x_{\nu}\}$ (appearing in the
previous statements) always contain essentially 
optimal number of sampling points (see also
\cite{Pes04a}).

\begin{theorem}
 If the  constant $c_{0}(M, L)>0$ is the same
as above, then  for any $\omega>0$ and $\rho=c_{0}\omega^{-1/2}$,
there exist $C_{1}(M, L), C_{2}(M, L)$ such
that the number of points in  any $\rho$-lattice $M_{\rho}$
satisfies the following inequalities
\begin{equation}
C_{1}\omega^{n/2}\leq |M_{\rho}|\leq
C_{2}\omega^{n/2};\label{rate}
\end{equation}

\end{theorem}\label{FT}
\begin{proof}
 According to the
definition of a lattice $M_{\rho}$ we have
$$
|M_{\rho}|\inf_{x\in M}Vol(B(x,\rho/4))\leq Vol(M)\leq
|M_{\rho}|\sup_{x\in M}Vol(B(x,\rho/2))
$$
or
$$
\frac{Vol(M)}{\sup_{x\in M}Vol(B(x,\rho/2))}\leq |M_{\rho}|\leq
\frac{Vol(M)}{\inf_{x\in M}Vol\left(B(x,\rho/4)\right)}.
$$
Since for  certain $c_{1}(M), c_{2}(M)$, all $x\in M$ and all
sufficiently small $\rho>0$, one has a double inequality
$$
c_{1}(M)\rho^{n}\leq  Vol(B(x,\rho))\leq c_{2}(M)\rho^{n},
$$
and since $\rho=c_{0}\omega^{-1/2}, $ we obtain that for certain
$C_{1}(M, L),\> C_{2}(M, L)$ and all $\omega>0$
\begin{equation}
C_{1}\omega^{n/2}\leq |M_{\rho}|\leq C_{2}\omega^{n/2}.
\end{equation}
\end{proof}

Note, that since the
operator $L$ is of order two, the dimension
$\mathcal{N}_{\omega}$ of the space ${\mathbf E}_{\omega}(L)$ is
given asymptotically by Weyl's formula \cite{Sogge93}
\begin{equation}
\mathcal{N}_{\omega}(M)\asymp C(M)\omega^{n/2},\label{W}
\end{equation}
where $n=\dim M$.

Since the inequalities (\ref{rate}) are in an agreement with 
Weyl's formula (\ref{W}),  the last theorem shows that
if $\omega$ is large enough then every sampling set $M_{\rho}$ for
${\mathbf E}_{\omega}(L)$ contains essentially the "correct" number of
points.

\subsection{Reconstruction using variational splines}

As it was explained (see the  formula (\ref{recon}))  one can always  use a dual frame for reconstruction of a function from its projections. However, in general  it is not easy to construct a dual frame (unless the frame is tight and then it is self-dual).

The goal of this section is to introduce variational splines on manifolds and to show that such splines can be used for reconstruction of  bandlimited functions from appropriate sets of samples.

Given a $\rho$ lattice $M_{\rho}=\{x_{\nu}\}$ and a sequence $\{z_{\nu}\}\in l_{2}$ we
will be
 interested in finding a
 function $s_{k}$ in the Sobolev space $ H_{2k}(M),$ where $k >n/2,\>\>\>n=\dim\>M,$  such that
\begin{enumerate}
\item $ s_{k}(x_{\nu})=z_{\nu}, \>\>\>x_{\nu}\in \ M_{\rho};$

\item function $s_{k}$ minimizes functional $g\rightarrow \|L^{k}g\|_{L_{2}(M)}$.
\end{enumerate}

For a given sequence  \  $\{z_{\nu} \}\in l_{2}$ consider a function $f$
from $H_{2k}(M)$ such that $f(x_{\nu})=z_{\nu}.$ Let $\mathcal{P}f$
 denote the orthogonal projection of this function $f$  in the Hilbert
space $H_{2k}(M)$ with the  inner product
$$\left<f,g\right>=\sum_{x_{\nu}\in
M_{\rho}}f(x_{\nu})g(x_{\nu})+ \left<\left(-L\right)^{k/2}f, \left(-L\right)^{k/2}g\right>$$
on the subspace
$U^{2k}(M_{\rho})=\left \{f\in H_{2k}(M)|f(x_{\nu})=0\right \}$ with the norm generated 
by the same inner product.
Then the function $g=f-\mathcal{P}f$ will be the unique solution of the
above minimization problem for the
 functional $g\rightarrow \|L^{k}g\|_{L_{2}(M)},\>\>\>
 k=2^{l}n,\>\>l=0,1, 2,...$.

  It is convenient to introduce the so-called Lagrangian splines. For a point $x_{\nu}$ in a lattice $M_{\rho}$ the corresponding Lagrangian spline $l^{2k}_{\nu}$ is a function  in $ H_{2k}(M)$ that minimizes the
same functional and
 takes value $1$ at the point $x_{\nu}$
and $0$ at all other points of $M_{\rho}$.
 Different parts of the following theorem can be found in \cite{Pes00}-\cite{Pes04c}.

\begin{theorem} The following statements hold:

\begin{enumerate}

  \item for any function $f$ from $H_{2k}(M), \>\>\>k=2^{l}n, \>\>l=1,2,
..., \>\>n=\dim M,$ there exists a unique
function $s_{k}(f)$ from the Sobolev space $H_{2k}( M), $ such that
$ f|_{M_{\rho}}=s_{k}(f)|_{M_{\rho}}; $ and this function 
  $s_{k}(f)$ minimizes the functional $u\rightarrow \|L^{k}u\|_{L_{2}( M)}$;

\item  every such function $s_{k}(f)$ is of the form 
$$
s_{k}(f)=\sum_{x_{\nu}\in M_{\rho}}
f(x_{\nu})l^{2k}_{\nu};
$$

\item   functions $l^{2k}_{\nu}$
form a Riesz basis in the space
 of all polyharmonic functions with singularities on $M_{\rho}$ i.e.  in the
space of such functions from
 $H_{2k}( M )$ which in the sense of distributions satisfy equation

$$L ^{2k}u=\sum_{x_{\nu }\in M_{\rho}}\alpha _{\nu }\delta (x_{\nu
})$$ where $\delta (x_{\nu})$ is the Dirac measure at the point
$x_{\nu }$; 

\item   if in addition the
set $M_{\rho}$ is invariant under some subgroup of diffeomorphisms acting on
$M$ then every two functions $l^{2k}_{\nu},\>\> l^{2k}_{\mu}$
 are translates of each other.
 \end{enumerate}
 \label{Splines}
 \end{theorem}

Next, if $f\in H_{2k}( M), k=2^{l}n, l=0,1,...$ then
 $f-s_{k}(f)\in U^{2k}(M_{\rho})$ and we have for
$k=2^{l}
n, l=0,1,...$
  $$\|f-s_{k}(f)\|_{L_{2}( M)}\leq
 (C_{0}\rho)^{k}\|L^{k/2}(f-s_{k}(f))\|_{L_{2}(M)}.$$
Using minimization property of 
$$
s_{k}(f)=\sum_{x_{\nu}\in M_{\rho}}f(x_{\nu})l^{2k}_{\nu}
$$ 
we obtain
the inequality 
\begin{equation}
\left\|f-\sum_{x_{\nu}\in M_{\rho}}f(x{_\nu})l^{2k}_{\nu}\right\|_{L_{2}(M)}\leq (c_{0}\rho)^{k}\|L^{k/2}f\|_{L_{2}( M)}, k=2^{l}n,\ l=0,1,...,\label{SobAppr}
\end{equation}
and for $f\in {\bf E}_{\omega}(L)$ the Bernstein inequality gives 
for any $f\in {\bf E}_{\omega}(L)$ and $ k=2^{l}n, \ l=0,1,....$,
\begin{equation}
\left\|f-\sum_{x_{\nu}\in M_{\rho}} f(x{_\nu})l^{2k}_{\nu}\right\|_{L_{2}( M)}\leq(c_{0}\rho\sqrt{\omega} )^{k}\|f\|_{L_{2}( M)}.\label{PWAppr}
\end{equation}

These inequalities lead to the following Approximation  and Reconstruction Theorem.

\begin{theorem}There exist constants $C=C(M)>0$ and  $c_{0}=c_{0}(M)>0$  such that for any $\omega>0$ and any $M_{\rho}$ with 
   $0<\rho \leq c_{0}\omega^{-1/2}$ the following inequality holds for all $f\in {\bf E}_{\omega}(L)$
$$
\sup_{x\in M}|(s_{k}(f)(x)-f(x))|\leq \omega^{n}
\left(C(M)\rho^{2}\omega\right)^{k-n}\|
f\|_{L_{2}(M)}, \>k=2^{l}n,\> l=0, 1, ... .
$$
In other words, by choosing  $\rho>0$ such that
$$
\rho<\left(C(M)\omega   \right)^{-1/2},
$$ 
one obtains  the following reconstruction algorithm  
$$
f(x)=\lim_{k\rightarrow \infty} s_{k}(f)(x),
$$
where convergence holds in the uniform norm.
\end{theorem}

It should be noted that there exists an algorithm \cite{Pes04c} which allows to express variational splines in terms of eigenfunctions of the operator $L$. Moreover, it was also shown \cite{Pes04a} that eigenfunctions of $L$ that belong to a fixed space ${\bf E}_{\omega}(L)$  can be perfectly approximated by  eigenfunctions of certain finite-dimensional matrices in  spaces of splines with a fixed set of nodes.

\subsection{Reconstruction using iterations}

For  a metric lattice $M_{\rho}=\{x_{\nu}\}$ one can construct
corresponding partitions of unity (\cite{Pes00}, \cite{Pes04a}) $\theta=\{\theta_{\nu}\}$ (we will
call it {\it associated} to the metric lattice $M_{\rho}$) with the
properties
\begin{enumerate}
\item $\operatorname{supp}(\theta_{\nu}) \subseteq B(x_{\nu},\rho/2),$
\item for each multi-index $\alpha$  there exists
$C_{\theta}(\alpha)$ such that
$$ \sup_\nu \sup_{x}|\partial^{|\alpha|}\theta_{\nu}|\leq C_{\theta}(\alpha). $$

\end{enumerate}
In analogy to established terminology for the group case we call
such a partition a {\it bounded uniform partition of unity} (for
short a BUPU).
For any continuous function  $f$ we define the  (Voronoi-type) function
$V_{M_{\rho}}f$  by:
\begin{equation}
V_{M_{\rho}}(f)=\sum_{\nu}f(x_{\nu}) \ \theta_{\nu}, \quad  \theta_{\nu}\in
C_{0}^{\infty}(B(x_{\nu},\rho/2)).
\end{equation}
For the case that $\theta_\nu$ is just the indicator function of those
points which are closer to $x_\nu$ than to any other point the resulting
function is constant on Voronoi domains and is simply the nearest
neighborhood interpolator.

It can be shown \cite{FP04} that for
any $\rho$-lattice $M_{\rho}$ the function $V_{M_{\rho}}(f)$ belongs to $L_{2}(M)$
as long as $f$ belongs to a $H^{k}(M)$, for some $k>n/2, n=\dim M$, i.e.\ that
$V_{M_{\rho}}$ is a linear operator from $H^{k}(M)$ into $L^2(M)$.

 The  approximation  operator $A$ will then be defined as
\begin{equation}
A_{M_{\rho}} =P_{\omega}\circ V_{M_{\rho}} , \quad \quad  \mbox{resp.\ } \quad A_{M_{\rho}}
f = P_{\omega} V_{M_{\rho}}(f).
\end{equation}
where $P_{\omega}$ is the orthonormal projector from $L_{2}(M)$
onto ${\bf E}_{\omega}(L)$.
Note  that
operator $A_{M_{\rho}}$ depends not just on the lattice $M_{\rho}$ but also on the
corresponding partition of unity. The following statement was proved in \cite{FP04}.
 
\begin{theorem}
For a given $k>n/2, \>n=\dim M,$ there exist constants $C=C(M, L, k)>0$ and $\rho(M, L, k)>0$
such that for any metric $\rho$-lattice $M_{\rho}$ with $ \rho<\rho(M, L, k)$, and any $ \omega > 0 $
\begin{equation}
\|f-A_{M_{\rho}}f\|_{L_{2}(M)} \leq C\rho(1+\omega^{2})^{k/2} \  \|f\|_{L_{2}(M)}  \quad
\mbox{ for all} \ f \in  {\bf E}_{\omega}(L).
\end{equation}
Hence, if $k>n/2$ and $\omega > 0 $ are given we can  choose
$$ \rho< \left(C(1+\omega^{2})^{k/2}\right)^{-1}, $$
which implies that
$$
\|f-A_{M_{\rho}}f\|_{L_{2}(M)}  \ \leq \  \epsilon \ \|f\|_{L_{2}(M)} \ \mbox{with}  \  \epsilon =
C\rho(1+\omega^{2})^{k/2}<1. $$
\end{theorem}
As a consequence we obtain that $f\in {\bf E}_{\omega}(L)$ can be
recovered from $A_{M_{\rho}}f$ by the following iterative procedure.
Starting from
 $f_{0}=A_{M_{\rho}}f$ and defining inductively
 \begin{equation} f_{m+1}=f_{m}+A_{M_{\rho}}(f-f_{m})
\end{equation}
one has
\begin{equation}
\lim_{m\rightarrow\infty}f_{m}=f
\end{equation}
with the error estimate
\begin{equation}
\|f-f_{m}\|_{L_{2}(M)}\leq\epsilon^{m+1}\|f\|_{L_{2}(M)}.
\end{equation}
Using Sobolev Embedding Theorems  one can show (\cite{Pes04a}, \cite{FP04}) that convergence of iterations takes place not only in
the $L_{2}(M)$ norm but also in Sobolev and uniform norms on the manifold.

\subsection{Reconstruction using the frame algorithm}

 What follows is a brief description of the frame algorithm (see \cite{Gr}).
Let $\{e_{\nu}\}$ be a frame in a Hilbert $H$ space with frames bounds $A, B$, i. e. 
$$
A\|f\|_{L_{2}(M)}^{2}\leq \sum_{\nu}|\left<f,e_{\nu}\right>|^{2}\leq B\|f\|_{L_{2}(M)}^{2}, \>\>f\in H.
$$
Given a relaxation parameter $0<\gamma<\frac{2}{B}$, set $\eta=\max\{|1-\gamma A|,\> |1-\gamma B|\}<1$. Let $f_{0}=0$ and define recursively
\begin{equation}
\label{rec}
f_{m}=f_{m-1}+\gamma S(f-f_{m-1}),
\end{equation}
where $S$ is the frame operator which is defined on $H$ by the formula
$$
Sf=\sum_{\nu}\left<f,e_{\nu}\right>e_{\nu}.
$$
In particular, $f_{1}=\gamma Sf=\gamma\sum_{\nu}\left<f, e_{\nu}\right> e_{\nu}$.  Pick $\gamma=\frac{2}{A+B}$.   Then  $\lim_{m\rightarrow \infty}f_{m}=f$ with a geometric rate of convergence, that is, 
\begin{equation}
\label{conv}
\|f-f_{m}\|_{L_{2}(M)}\leq \eta^{m}\|f\|_{L_{2}(M)},
\end{equation}
where 
\begin{equation}
\label{convfac}
\eta=\frac{B-A}{A+B} \le \frac{B/A-1}{2}.
\end{equation}

Let $\psi_{\nu}$ be a projection of the Dirac measure $\delta_{x_{\nu}},\>\>\>x_{\nu}\in  M_{\rho},$ onto space ${\bf E}_{\omega}(L)$. The Plancherel-Polya inequality (\ref{completePlPo100}) shows that $\{\psi_{\nu}\}$ is a frame in  the space ${\bf E}_{\omega}(L)$ and the corresponding frame constants are $ A\rho^{-n/2}$ and $B\rho^{-n/2}$.

\section{Generalized variational splines  on compact Riemannian manifolds with applications to integral geometry}\label{splines}

\subsection{Generalized variational splines  on compact Riemannian manifolds}\label{splines-1} 

We still consider a compact Riemannian manifold $M$ and a differential second order elliptic operator $-L$ which is self-adjoint and non-negative in the natural space $L_{2}(M)$.
For a given finite family of pairwise different submanifolds $\{\mathcal{M}_{\nu}\}_{1}^{N}$ consider the following family of  distributions
 \begin{equation}
 \label{functionals}
 F_{\nu}(f)=\int_{\mathcal{M}_{\nu}}fd\mu_{\nu}
 \end{equation}
 ($d\mu_{\nu}$ is a measure on $\mathcal{M}_{\nu}$) which are well defined at least for functions in $H_{t}(M)$ with $t>n/2$.
 In particular, if every $\mathcal{M}_{\nu}$ is a point $x_{\nu}\in M$, then every $F_{\nu}$ is a Dirac measure $\delta_{x_{\nu}}\>\>\nu=1,...,N,\>\>x_{\nu}\in M.$
 
Given a sequence of complex numbers $v=\{v_{\nu}\},$ $ \nu=1,2,...,N,$ and a $t>n/2$ we consider the
following 

{\bf Variational problem}.

\textsl{Find a function  $u$ from the space $H_{t}(M),\>\>t>n/2,$ which has
the following properties:} \label{var_prob}

\begin{enumerate}

\item $ F_{\nu}(u)=v_{\nu}, \>\>\>\nu=1,2,...,N,\>\>\> v=\{v_{\nu}\},$

\item  $u$ \textsl{minimizes functional $u\rightarrow \|(I-L)
^{t/2}u\|_{L_{2}(M)}$.} 

\end{enumerate}

One can show \cite{Pes04c} that  solution to Variational problem exists and is
unique.  
The following  Independence  Assumption which first appeared in     \cite{Pes04c}      is necessary in order to determine explicit form of the solution.

 \textbf{Independence Assumption.} \textsl{There are functions
  $\vartheta_{\nu}\in C^{\infty}(M)$ such
that}
\begin{equation}
\label{independence}
F_{\nu}(\vartheta_{\mu})=\delta_{\nu\mu},
\end{equation}
\textsl{where $\delta_{\nu\mu}$ is the Kronecker delta.}

Note, that this assumption implies in particular that the
functionals $F_{\nu}$ are linearly independent. Indeed, if for certain coefficients $\gamma_{1},\gamma_{2}, ...,
\gamma_{N}$ we have a relation 
$
\sum _{\nu=1}^{N}\gamma_{\nu}F_{\nu}=0,
$
then for any $1\leq\mu\leq N$ we obtain that 
$
0=\sum_{\nu=1}^{N}\gamma_{\nu}F_{\nu}(\vartheta_{\mu})=\gamma_{\mu}.
$

The families of distributions that satisfy our condition
  include in particular finite families of $\delta$ functionals and their
  derivatives. Another example is a set  of integrals over submanifolds from a   finite family of
submanifolds of any codimension.

 The solution to the Variational
Problem will be called a spline and will be denoted as $s_{t}(v).$
 The set of all solutions for a fixed set of distributions
$F=\{F_{\nu}\}$ and a fixed $t$ will be denoted as $S(F,t).$
\begin{definition}
\label{interpolationdef}
Given a function $f\in H_{t}(M)$ we will say that 
spline $s\in S(F,t)$ interpolates $f$  if
$$
F_{\nu}(f)=F_{\nu}(s).
$$
\end{definition}
The interpolating  spline exists and unique (see below) and  will be denoted as $s_{t}(f).$
Note, that from the point of view of the classical theory of variational
 splines it would be more natural to consider minimization of the
functional
$
u\rightarrow \|L^{t/2}u\|.
$
However, in the case of a general compact manifolds it is easier to
work with the operator $I-L$ since this operator is
invertible. 

Our main result concerning  variational splines is the following (see \cite{Pes04c}).
\begin{theorem}
\label{MainTheorem}
If $t>n/2$, then for any given sequence of scalars 
$v=\{v_{\nu} \}, \>\>\nu=1,2,...N,$ the following statements are
equivalent:

\begin{enumerate}

\item $s_{t}(v)$ is the solution to \textsl{the Variational Problem};

\item  $s_{t}(v)$ satisfies the  the following equation in the sense of
distributions
\begin{equation}
\label{Main Equation-1}
(I-L)^{t}s_{t}(v)=\sum_{\nu=1}^{N}\alpha_{\nu}\overline{F_{\nu}},\>\>\>\alpha_{\nu}=\alpha_{\nu}(s_{t}(v)),\>\>
t>n/2,
\end{equation}
where $\alpha_{1},...,\alpha_{N}$ form a
solution of the $N\times N$ system
\begin{equation}
\label{eq:linsystem}
\sum_{\nu=1}^{N}\beta_{\nu\mu}\alpha_{\nu}=v_{\mu},\>\>\>\alpha_{\nu}=\alpha_{\nu}(s_{t}(v)),\>\>
\mu=1,...,N,
\end{equation}
and
\begin{equation}
\label{eq:linsolution}
\beta_{\nu\mu}=\sum_{j=0}^{\infty}(1+\lambda_{j})^{-t}\overline{F_{\nu}(u_{j})}
F_{\mu}(u_{j}),\>\>\>Lu_{j}=-\lambda_{j}u_{j};
\end{equation}

\item  the Fourier series of $s_{t}(v)$  has  the following form
\begin{equation}
\label{Fourier Series}
s_{t}(v)=\sum_{j=0}^{\infty}c_{j}(s_{t}(v))u_{j},
\end{equation}
where
$$
c_{j}(s_{t}(v))=\left<s_{t}(v),u_{j}\right>=(1+\lambda_{j})^{-t}
\sum_{\nu=1}^{N}\alpha_{\nu}(s_{t}(v))\overline{F_{\nu}(u_{j})}.
$$
\end{enumerate}
\end{theorem}

\begin{remark}
It is important to note that the system (\ref{eq:linsystem}) is always solvable
according to our uniqueness and existence result for the
Variational Problem.

\end{remark}

\begin{remark}
It is also necessary to note that the series (\ref{eq:linsolution}) is absolutely
convergent if $t$ is sufficiently large. Indeed, if all  functionals $F_{\nu}$ belong to $H_{-t_{0}}(M),\>t_{0}>0,$ 
 we obtain that
for any normalized eigenfunction $u_{j}$ which corresponds
to the eigenvalue $\lambda_{j}$ the following inequality holds
true for some $c=c(M,F)$

$$\left|F_{\nu}(u_{j})\right|\leq
c\|(1-L)^{t_{0}/2}u_{j}\|_{L_{2}(M)}\leq
c(1+\lambda_{j})^{t_{0}/2},\>\>\>F=\{F_{\nu}\}.
$$
So
$$
\left|\overline{F_{\nu}(u_{j})}F_{\mu}(u_{j})\right| \leq
c^{2}(1+\lambda_{j})^{t_{0}},
$$
and
$$
\left|(1+\lambda_{j})^{-t}\overline{F_{\nu}(u_{j})}F_{\mu}(u_{j})\right|\leq
c^{2}(1+\lambda_{j})^{(t_{0}-t)}.
$$
It is known  that the series
$$
\sum_{j}\lambda_{j}^{-\tau},
$$
which defines the $\zeta-$function of an elliptic second order
operator, converges if $\tau>n/2$. This implies absolute
convergence of (\ref{eq:linsolution}) in the case $t>t_{0}+n/2$.

\end{remark}

One can show  that splines provide an optimal approximations to sufficiently smooth  functions.  Namely let $Q(F,f,t,K)$ be the set  of all functions $h$  in
$H_{t}(M)$ such that

\begin{enumerate}

\item  $F_{\nu}(h)=F_{\nu}(f), \nu=1,2,...,N,$

\item  $\|h\|_{ H_{t}(\mathbb{S}^{n})}\leq K,$ for a real $ K\geq \|s_{t}(f)\|_{ H_{t}(\mathbb{S}^{n})}.$

\end{enumerate}

The set $Q(F,f,t,K)$ is  convex,  bounded and closed. 
The following theorem (see \cite{Pes04c}) shows that splines provide an optimal approximations to functions in $Q(F,f,t,K)$.

\begin{theorem}
\label{optim}
The spline $s_{t}(f)$ is the symmetry center of $Q(F,f,t,K)$. This means that for
any $h\in Q(F,f,t,K)$
\begin{equation}
\label{optim-ineq}
  \|s_{t}(f)-h\|_{ H_{t}(\mathbb{S}^{n})}\leq \frac{1}{2} diam \>Q(F,f,t,K).
\end{equation}
\end{theorem}

\subsection{The Funk-Radon transform on spheres}\label{Radon-spheres}

 We consider the unit sphere
$\mathbb{S}^{n}\subset \mathbb{R}^{n+1}$ and the corresponding space
$L_{2}(\mathbb{S}^{n})$ constructed with respect to normalized and
rotation-invariant measure. 

Let  $Y^{i}_{k}$ be an orthonormal basis of spherical
harmonics in the space $L_{2}(\mathbb{S}^{n})$, where $k=0,1, ... ; i=
1,2,..., d_{n}(k)$ and
$$
d_{n}(k)=(n+2k-1)\frac{(n+k-2)!}{k!(n-1)!}
$$
is the dimension of the subspace of spherical harmonics of degree
$k$. 
Note, that 
\begin{equation}\label{symm}
Y_{k}^{i}(-x)=(-1)^{k}Y_{k}^{i}(x).
\end{equation}
The Fourier decomposition of $f\in L_{2}(\mathbb{S}^{n})$ is
\begin{equation}\label{Fourier-Laplace}
f(x)=\sum_{i,k}c_{i,k}(f)Y^{i}_{k}(x),
\end{equation}
where
$$
c_{i,k}(f)=\int_{\mathbb{S}^{n}}f(x)\overline{Y^{i}_{k}(x)}dx=\left<f,Y_{k}^{i}\right>_{L_{2}(\mathbb{S}^{n})}.
$$
 To every function  $f\in L_{2}(\mathbb{S}^{n})$ the  Funk-Radon transform associates  its integrals over great subspheres:
$$
Rf(\theta^{\perp}\cap \mathbb{S}^{n})=\int_{\theta^{\perp}\cap \mathbb{S}^{n}}f(x)dx,
$$
where $ \theta^{\perp}\cap \mathbb{S}^{n}$ is the great subsphere of $\mathbb{S}^{n}$
whose plane has normal $\theta$.

In what follows we identify unit vector $\theta$ with a point on a "dual" unit sphere $\mathbb{S}^{n}$. 
Thus given a set of points on a dual sphere one can construct a corresponding set of subspheres $\theta^{\perp}\cap \mathbb{S}^{n}$. 

Given a subsphere of co-dimension one, we associate with it a pair of unit vectors (points) $\theta$ and $-\theta$  on the dual unit sphere.  So given a set of subsphere of co-dimension one we can associate with it a set of points $\Xi$ which is symmetric in the sense that $\Xi=-\Xi$. 

\begin{definition}
We will say that a set of equatorial
 subspheres $\{w_{\nu}\},\>\> \nu=1,2,...,N,$ is a $\rho$-lattice if the corresponding set of points on the dual sphere is a $\rho$-lattice $M_{\rho}$.
\end{definition}
 Note that by a previous discussion the set $M_{\rho}$ is symmetric : $M_{\rho}=-M_{\rho}$.

If a function $f\in L_{2}(\mathbb{S}^{n})$ has Fourier coefficients
$c_{i,k}(f)$ then its Radon Transform is given by the formula
$$
R(f)=\pi^{-1/2}\Gamma((n+1)/2)\sum_{i,k}r_{k}c_{i,k}(f)Y_{k}^{i}.
$$
where $Y_{k}^{i}$ are spherical harmonics  and
\begin{equation}\label{r}
r_{k}=(-1)^{k/2}\Gamma((k+1)/2))/\Gamma((k+n)/2)
\end{equation}
if $k$ is even and $r_{k}=0$ if $k$ is odd. It implies in particular that operators $\Delta$ and $R$ commute on a set of smooth functions.  A function $f \in L_{2}(\mathbb{S}^{n})$ is said to be even if its Fourier series (\ref{Fourier-Laplace}) contains only harmonics of even degrees $k=2m$.
Because the coefficients $r_{k}$ have asymptotics
$(-1)^{k/2}(k/2)^{(1-n)/2}$ as $k$ goes to infinity we have the following result.

\begin{theorem}\label{SRT}
The spherical Radon transform $R$ is a continuous operator from the Sobolev space of even functions
$H^{even}_{t}(\mathbb{S}^{n})$ onto the space $H^{even}_{t+(n-1)/2}(\mathbb{S}^{n})$.
Its inverse $R^{-1}$ is a continuous operator from the space
$H^{even}_{t+(n-1)/2}(\mathbb{S}^{n})$ onto the space $H^{even}_{t}(\mathbb{S}^{n})$.  If $f\in H^{even}_{t+(n-1)/2}(\mathbb{S}^{n})$ and it has Fourier series $f=\sum_{i, m}c_{i, 2m}(f)Y_{2m}^{i}$ then 
\begin{equation}\label{inverse}
R^{-1}f=\frac{\sqrt{\pi}}{\Gamma((d+1)/2)}\sum_{i,m}\frac{c_{i,2m}(f)}{r_{2m}}Y_{2m}^{i}.
\end{equation}

\end{theorem}

\subsection{Approximate inversion of the Funk-Radon transform using generalized splines}\label{inverF-R}

 We consider approximate inversion  of the Radon transform on $\mathbb{S}^{n}$ (see subsection \ref{Radon-spheres}) when only a finite number of integrals over equatorial subspheres is given.  
Let $\{w_{\nu}\}, \nu=1,2,...,N, $ be a finite set of equatorial
subspheres on $\mathbb{S}^{n}$ of codimension one (which play the same role as submanifolds $M_{\nu}$ before) and distributions
$F_{\nu}$ are given by
 formulas
 $$
 F_{\nu}(f)=\int_{w_{\nu}}fdx.
 $$
 By solving corresponding variational problem we can find a
spline $s_{t}(f)\in H_{t}(\mathbb{S}^{n})$ such that
$$
F_{\nu}(s_{t}(f))=F_{\nu}(f), \nu=1,2,...,N,
$$
and $s_{t}(f)$ minimizes norm $\left\|(1-\Delta)^{t/2}s_{t}(f)\right\|_{L_{2}(\mathbb{S}^{n})}$ where $\Delta$ is the Laplace-Beltrami operator in $L_{2}(\mathbb{S}^{n})$.

Our Theorem \ref{MainTheorem} in the case of  the spherical Radon transform is summarized in
the following statement.

\begin{theorem}\label{Main-Applic}
For a given $\rho$-lattice $W=\{w_{\nu}\}$ of
equatorial
 subspheres $\{w_{\nu}\}, \nu=1,2,...,N,$
an even smooth function $f$   and any $t>n/2$ define $s_{t}(f)$
by the formula
$$
s_{t}(f)=\sum_{i,k}c_{i,k}(s_{t}(f))Y^{i}_{k},
$$
where
$$
c_{i,k}(s_{t}(f))=(1+\lambda_{i,k})^{-t}\sum_{\nu=1}^{N}\alpha_{\nu}(s_{t}(f))
\int_{w_{\nu}}Y^{i}_{k}dx,
$$
and
$$
\sum_{\nu=1}^{N}b_{\nu\mu}\alpha_{\nu}(s_{t}(f))=v_{\mu},\>\>\>\>\>\>\>
v_{\mu}=\int_{w_{\mu}}fdx,\>\>\> \>\>\>\>\mu=1,2,...,N,
$$
$$b_{\nu\mu}=\sum_{i,k}(1+\lambda_{i,k})^{-t}\int_{w_{\nu}}Y^{i}_{k}dx
\int_{w_{\mu}}Y^{i}_{k}dx.
$$

The function $s_{t}(f)$   has the following properties.

\begin{enumerate}
\item The function $s_{t}(f)$ is even.
\item  Integrals of $s_{t}(f)$ over subspheres $w_{\nu}$ have
prescribed values $v_{\nu}$

$$
\int_{w_{\nu}}s_{t}(f)dx= v_{\nu}, \>\>\>\>\>\>\>\nu=1,2,...,N.
$$

\item  Among all functions that satisfy (2) function $s_{t}(f)$
minimizes
 the Sobolev norm
 $$
 \left\|(I-\Delta)^{t/2}s_{t}(f)\right\|_{L_{2}(\mathbb{S}^{n})}=\left(\sum_{\nu=1}^{N}\alpha_{\nu}(s_{t}(f))v_{\nu}\right)^{1/2}.
 $$

\item The function $s_{t}(f)$ is the center of the convex set
$Q(F,f,t,K)$ of all functions $h$ from $H_{t}(\mathbb{S}^{n})$ that satisfy both the condition 
(2) and the inequality
$$
\left\|(I-\Delta)^{t/2}h\right\|_{L_{2}(\mathbb{S}^{n})}\leq K,
$$
for any fixed $K\geq
\left(\sum_{\nu=1}^{N}\alpha_{\nu}(s_{t}(f))v_{\nu}\right)^{1/2}.$
In other words for any $h\in Q(F, f,t,K)$
$$
\left\|s_{t}(f)-h\right\|_{ H_{t}(\mathbb{S}^{n})}\leq\frac{1}{2}diam\> Q(F, f, t, K).
$$
\end{enumerate}

\end{theorem}

To prove item (1) one has to note (see \cite{Pes04c}) that because of (\ref{symm}) integrals of $Y_{k}^{i}$ over great circles are zero as long as $k$ is odd. It implies that $c_{i,k}(s_{t}(f))=0$ for every odd $k$ which means that  decomposition of the corresponding spline $s_{t}(f)$   contains only harmonics $Y_{k}^{i}$ with even $k$.

\subsection{ A sampling theorem for the  Funk-Radon transform of bandlimited functions on $\mathbb{S}^{n}$}\label{pointwise}

According to  Theorem \ref{SRT} if $f\in H_{t}^{even}(\mathbb{S}^{n})$ then $R f\in H_{t+(n-1)/2}^{even}(\mathbb{S}^{n})$. 
Integral of $f$ over a great subsphere  $w_{\nu}$ is the value of $R f$ at points $x_{\nu}$ and $-x_{\nu}$ which are dual to $w_{\nu}$ and belong to the dual sphere.

 For a fixed integer $l\geq 0$ and $t>n/2$ we  apply Theorem \ref{MainTheorem} to the manifold $\mathbb{S}^{n}$ and the set of distributions $\{F_{\nu}\}=\{(\delta_{x_{\nu}}+\delta_{-x_{\nu}})\}$,  where
$
\delta_{x_{\nu}}$ is the Dirac measure. 
 This way we  construct  the spline  $s_{\tau}(R f)$ with
$
\tau=2^{l}n+t+(n-1)/2
$
 which interpolates $Rf$ on $M_{\rho}=\{x_{\nu}\}$ and minimizes functional \begin{equation}\label{functional20}
u\rightarrow \|(1- \Delta)^{\tau/2}u\|_{L_{2}(\mathbb{S}^{n})},
\end{equation}
where $\Delta$ is the Laplace-Beltrami operator in $L_{2}(\mathbb{S}^{n})$.
 Fourier coefficients of $s_{\tau}(R f)$ with respect to the basis $\left\{Y_{k}^{i}\right\}$ can be obtained by using Theorem \ref{MainTheorem}.  
 
 It is clear 
that $s_{\tau}(R f)$ has a representation of the form
\begin{equation}\label{shat}
s_{\tau}(Rf)(x)=\sum_{k}\sum_{i} c_{i}^{2k}(R f)Y_{2k}^{i}(x),\>\>\>x\in \mathbb{S}^{n},
\end{equation}
where
$
c_{i}^{k}(R f)
$
are Fourier coefficients of $s_{\tau}(Rf)$. 
Applying (\ref{inverse}) we obtain that the following function defined on $\mathbb{S}^{n}$ 
\begin{equation}\label{bigS}
S_{\tau}(f)=R^{-1}\left(s_{\tau}(R f)\right)
\end{equation}
has a representation 
$$
S_{\tau}(f)=\frac{\sqrt{\pi}}{\Gamma((n+1)/2)}\sum_{k}\sum_{i} \frac{c^{i}_{2k}(R f)}{r_{2k}}Y_{2k}^{i},
$$
where $r_{2k}$ are  defined in  (\ref{r}).

In the following theorem we assume that a $\rho$-lattice $M_{\rho}=\{x_{\nu}\}$ is a subset of points  on the sphere $\mathbb{S}^{n}$ and $M_{\rho}$  is dual to a collection of great subspheres $\{w_{\nu}\}$. We also assume that functions $S_{\tau}(f)$ constructed using vales of $Rf$ on $M_{\rho}$.
One can prove the following Theorem (see \cite{Pes04c}).

\begin{theorem}\label{last-thm555}
If $t>n/2$ then there exists a constant $C=C(n,t)>0$ such that for  any $\rho$-lattice  $M_{\rho}=\{x_{\nu}\}\subset \mathbb{S}^{n}$ with sufficiently small $\rho>0$  and any sufficiently  smooth
function $f$ on $\mathbb{S}^{n}$ the following inequality holds true
\begin{equation}\label{Th62}
\left\|\left(S_{\tau}(f)-f\right)\right\|_{H_{t}(\mathbb{S}^{n})}\leq
 2(C\rho^{2n})^{2^{l-1}}\|Rf\|_{H_{\tau}(\mathbb{S}^{n})},\>\>\>\tau=2^{l}n+t+(n-1)/2,
\end{equation}
for any $l=0, 1, ... \>\>.\>\>\>$ In particular, if a natural $k$ satisfies the inequality $t>k+n/2$, then 
\begin{equation}
\|S_{\tau}( f)-  f\|_{C^{k}(\mathbb{S}^{n})}\leq 
2\left(C\rho^{2n}\right)^{ 2^{l-1}} \|R f\|_{H_{\tau}(\mathbb{S}^{n})}
\end{equation}
for any $l=0, 1, ... .$ 
\end{theorem}

For an $\omega>0$ let us consider the subspace ${\bf E}_{\omega}^{even}(\mathbb{S}^{n})$ of even $\omega$-bandlimited functions on $\mathbb{S}^{n}$. Clearly, this subspace is invariant under $R$. 

Note, that for functions in ${\bf E}_{\omega}^{even}(\mathbb{S}^{n})$
 the following Bernstein-type inequality holds 
 $$
\|(I-\Delta)^{s}Rf\|_{L^{2}(\mathbb{S}^{n})}  \leq (1+\omega)^{s}\|Rf\|_{L^{2}(\mathbb{S}^{n})} .
$$
As a consequence of the previous Theorem we obtain the next one (see \cite{Pes04c}).
\begin{theorem}(Sampling Theorem For Radon Transform)\label{SSTh}.
\label{SamplingT}
If $t>n/2$ then there exist  constant $C=C(n, t)>0$ such that for  any $\rho$-lattice  $M_{\rho}=\{x_{\nu}\}\subset \mathbb{S}^{n}$ with sufficiently small $\rho>0$  and any  $f\in {\bf E}_{\omega}^{even}(\mathbb{S}^{n})$
  one has the estimate for $\tau=2^{l}n+t+(n-1)/2$
\begin{equation}
\|\left(S_{\tau}(f)-f\right)\|_{H_{t}(\mathbb{S}^{n})}\leq
$$
$$
2(1+\omega)^{t/2+(n-1)/2}\left(C\rho^{2}(1+\omega)\right)^{2^{l-1}n} \| Rf\|_{L_{2}(\mathbb{S}^{n})},
 \end{equation}
for any $l=0, 1, ... \>\>.\>\>\>$ In particular, if a natural $k$ satisfies the inequality $t>k+n/2$, then for  $\tau=2^{l}n+t+(n-1)/2$
\begin{equation}
\|S_{\tau}( f)-  f\|_{C^{k}(\mathbb{S}^{n})}\leq
$$
$$
 2(1+\omega)^{t/2+(n-1)/2}\left(C\rho^{2}(1+\omega)\right)^{2^{l-1}n} \| Rf\|_{L_{2}(\mathbb{S}^{n})},
\end{equation}
for any $l=0, 1, ... .$ 
\end{theorem}
This Theorem \ref{SSTh} shows that if 
$$
\rho<\left(C(1+\omega)\right)^{-1/2}
$$
then every $f\in {\bf E}_{\omega}^{even}(\mathbb{S}^{n})$ is completely determined by a finite  set of values 
$$
Rf(x_{\nu})=\int_{w_{\nu}} f dx,\>\>\>\>\>\>x_{\nu}\in M_{\rho}.
$$
Moreover, it shows that $f$  can be reconstructed as a limit (when $l$ goes to infinity) of functions $S_{2^{l}n+t+(n-1)/2}( f)$ which were  constructed by using only the set of values  $\left\{R f(x_{\nu})\right\}$ of  the Radon transform $R f$.

\subsection{Hemispherical Radon transform on $\mathbb{S}^{n}$}\label{hem}

 To every function  $f\in L_{2}(\mathbb{S}^{n})$ the hemispherical transform $T$
 assigns a function $Tf\in L_{2}(\mathbb{S}^{n})$ on the dual sphere $\mathbb{S}^{n}$
 which is given
by the formula
$$
(Tf)(\xi)=\int_{\xi\cdot x >0} f(x)dx.
$$

For every function $f\in L_{2}(\mathbb{S}^{n})$ that has Fourier
coefficients $c_{i,j}(f)$ the hemispherical transform can be given
explicitly by the formula
$$
Tf=\pi^{(n-1)/2}\sum_{i,k}m_{k}c^{i}_{k}(f)Y^{i}_{k},
$$
where $m_{k}=0, $ if $k$ is even and
$$
m_{k}=(-1)^{(k-1)/2}\frac {\Gamma(k/2)}{\Gamma((k+n+1)/2))},
$$
 if $k$ is odd.

The transformation $T$ is one to one on the subspace of odd
functions (i.e. $f(x)=-f(-x))$ of a Sobolev space
$H_{t}^{odd}(\mathbb{S}^{n})$ and maps it continuously onto
$H_{t+(n+1)/2}^{odd}(\mathbb{S}^{n})$,
$$
T\left(H_{t}^{odd}(\mathbb{S}^{n})\right)=H_{t+(n+1)/2}^{odd}(\mathbb{S}^{n}).
$$

\subsection{Approximate inversion of the hemispherical Radon transform on $\mathbb{S}^{n}$.}\label{inverHem}

Let $\{h_{\nu}\}, \nu=1,2,...,N, $ be a finite set of hemispheres
on $\mathbb{S}^{n}$. We consider functionals $F_{\nu}$ on $L_{2}(\mathbb{S}^{n})$
which are given by
 formulas
 $$
F_{h_{\nu}}=F_{\nu}=\int_{h_{\nu}}fdx.
 $$

We will assume that the set of points $\Xi=\{\xi_{\nu}\}$ on the
dual sphere $\mathbb{S}_{*}^{n}$ that corresponds to the set of hemispheres
$h_{\nu}$ is symmetric in the sense that $\Xi=-\Xi$.
  Under this assumption we choose a $t>0$  and an odd function
 $f$ and consider the following variational problem: find a
 function $s_{t}(f)\in H_{t}(\mathbb{S}^{n}), t>0$ such that

1) $ F_{\nu}(s_{t}(f))=F_{\nu}(f), \nu=1,2,...,N,$

2) $s_{t}(f)$ minimizes norm $\|(I-\Delta)^{t/2}s_{t}(f)\|_{L_{2}(\mathbb{S}^{n})}$.

Since $\Xi=-\Xi$ and function $f$ is odd, the solution $s_{t}(f)$
will be an odd function.

According to  Theorem \ref{Main-Applic}  the Fourier series of $s_{t}(f)$ is
$$
s_{t}(f)=\sum_{i,j}c_{i,j}(s_{t}(f))Y^{i}_{j},
$$
where the Fourier coefficients $c_{i,j}(s_{t}(f))$ of $s_{t}(f)$
are given by formulas
$$
c_{i,j}(s_{t}(f))=\left<s_{t}(f),Y^{i}_{j}\right>=(1+\lambda_{i,j})^{-t}\sum_{\nu=1}^{N}
\alpha_{\nu}(s_{t}(f))\int_{h_{\nu}}Y^{i}_{j}dx,
$$
where vector $\alpha(s_{t}(f))$ is the solution of the following
$N\times N$ system
$$
\sum_{\nu=1}^{N}b_{\nu\mu}\alpha_{\nu}(s_{t}(f))=\int_{h_{\mu}}f
dx, \mu=1, 2, ... N,
$$
where
$$
b_{\nu\mu}=\sum_{i,j}(1+\lambda_{i,j})^{-t}
\int_{h_{\nu}}Y^{i}_{j} dx\int_{h_{\mu}}Y^{i}_{j} dx.
$$

This spline provides the optimal approximation to $f$ in the sense
that it is the center of the convex set $Q(F, f, t, K)$ of all
functions $\psi$ from $H_{t}(\mathbb{S}^{n})$ that satisfy

\begin{equation}
\int_{h_{\nu}}\psi dx=\int_{h_{\nu}}f dx
\end{equation}
and the inequality
\begin{equation}
\|(I-\Delta)^{t/2}\psi\|_{L_{2}(\mathbb{S}^{n})}\leq K
\end{equation}
for any fixed $K$ that satisfies the inequality
$$
K\geq
\|s_{t}(f)\|_{ H_{t}(\mathbb{S}^{n})}=\left(\sum_{\nu=1}^{N}\alpha_{\nu}(s_{t}(f))\int_{h_{\nu}}f
dx\right)^{1/2}.
 $$

Our results about the hemispherical transform are summarized in
the following theorem.

\begin{theorem}
For a given symmetric set $H=\{h_{\nu}\}, \nu=1,2,...,N,$ of
hemispheres $h_{\nu}$, an odd function $f$ and any $t>0$ define
the function $s_{t}(f)$ by the formula
$$
s_{t}(f)=\sum_{i,j}c_{i,j}(s_{t}(f))Y^{i}_{j},
$$
where
$$
c_{i,j}(s_{t}(f))=(1+\lambda_{i,j})^{-t}\sum_{\nu}^{N}\alpha_{\nu}(s_{t}(f))
\int_{h_{\nu}}Y^{i}_{j}dx,
$$
and
$$
\sum_{\nu=1}^{N}b_{\nu\mu}\alpha_{\nu}(s_{t}(f))=v_{\mu},\>\>\>\>\>\>\>\>\>\>\>
v_{\mu}=\int_{h_{\mu}}fdx, \>\>\>\>\>\>\>\>\mu=1,2,...,N,
$$
$$b_{\nu\mu}=\sum_{i,j}(1+\lambda_{i,j})^{-t}\int_{h_{\nu}}Y^{i}_{j}dx
\int_{h_{\mu}}Y^{i}_{j}dx.
$$

The function $s_{t}(f)$ is odd and it has the following
properties.

\begin{enumerate}

\item  Integrals of the function $s_{t}(f)$ over hemispheres $h_{\nu}$
have  prescribed values $v_{\nu}$:

$$
\left(Ts_{t}(f)\right)(\xi_{\nu})=\int_{h_{\nu}}s_{t}(f)dx=
v_{\nu}, \>\>\>\>\>\>\>\>\>\>\>\nu=1,2,...,N.
$$

\item  Among all functions that satisfy (5.1) function $s_{t}(f)$
minimizes
 the Sobolev norm
 $$ \left\|(I-\Delta)^{t/2}s_{t}(f)\right\|_{L_{2}(\mathbb{S}^{n})}=
 \left(\sum_{\nu=1}^{N}\alpha_{\nu}(s_{t}(f))v_{\nu}\right)^{1/2}.$$

\item  Function $s_{t}(f)$ is the center of the convex set $Q(F,
f,t,K)$ of all functions  $g$  from $H_{t}(\mathbb{S}^{n})$ that satisfy
(5.1) and the inequality
\begin{equation}
\left\|(I-\Delta)^{t/2}g\right\|_{L_{2}(\mathbb{S}^{n})}\leq K,
\end{equation}
for any fixed $K\geq
\left(\sum_{\nu=1}^{N}\alpha_{\nu}(s_{t}(f))v_{\nu}\right)^{1/2}.$
In other words for any $g\in Q(F,f,t,K)$
$$
\|s_{t}(f)-g\|_{ H_{t}(\mathbb{S}^{n})}\leq\frac{1}{2}diam Q(F, f, t, K).
$$
\end{enumerate}
\end{theorem}

\subsection {A sampling theorem for the hemispherical Radon transform on $\mathbb{S}^{n}$.}\label{samplHem}

Applying our Approximation Theorem we obtain the following result
about convergence of interpolants in the case of hemispherical
transform. In this Theorem we use the following parameter $\rho$
$$
\rho=\sup_{\nu}\inf_{\mu} dist(\xi_{\nu},\xi_{\mu}),
\xi_{\nu},\xi_{\mu}\in \Xi, \nu\neq\mu,
$$
as a measure of the density of the set $\Xi$.
\begin{theorem}
There exists a constant $C$ such that  for any $l=0,1,..., $ any
$k<2^{l}n-n/2$ and for any odd function  $f\in H_{k}(\mathbb{S}^{n})$ we have
$$
\|s_{2^{l}n}(f)-f\|_{ H_{k}(\mathbb{S}^{n})}\leq
(C\rho^{2})^{2^{l}n}\|(I-\Delta)^{\tau/2}Tf\|_{L_{2}(\mathbb{S}^{n})}, \>\>l=0,1,...
,\>\>\tau=2^{l}n+(n+1)/2,
$$
and if $f\in {\bf E}_{\omega}^{odd}(S^{n})$ then
$$
\|s_{2^{l}n}(f)-f\|_{ H_{k}(\mathbb{S}^{n})}\leq
(C\rho^{2}(1+\omega))^{2^{l}n}(1+\omega)^{(n+1)/2}\|Tf\|_{L_{2}(\mathbb{S}^{n})},\>\>
l=0,1,... .
$$
\end{theorem}

The first inequality shows that for any odd smooth function $f$
the interpolants $s_{2^{l}n}(f)$ of a fixed order $2^{l}n,\>\>
l=0,1,..., \>\>k<2^{l}n-n/2,$ converge to $f$ in the uniform norm
$C^{k}(M)$ as long as the parameter $\rho$ goes to zero, i.e. the
set $\Xi$ on the dual sphere gets denser.

 The second inequality in the Theorem shows,
that interpolants converge to an odd harmonic polynomial of order
$\omega$ for a fixed set of hemispheres $\Xi$ if
$C\rho^{2}(1+\omega)<1$ and if  $l$ goes to infinity. This
statement is an analog of the Sampling Theorem for the
hemispherical transform.

\subsection{Radon transform on the group of rotations  $SO(3)$}\label{RadSO}
The following information can be found in \cite{GMS}, \cite{Vilenkin}. 
The group  of rotations $SO(3)$ of $\mathbb{R}^{3}$ consists of $3\times 3$ real matrices $U$ such that  $U^TU = I,\>\>\ {\rm det\,}U = 1$.
It is known that any $g\in SO(3)$ has a unique representation of the form
$$ 
g = Z(\gamma)X(\beta)Z(\alpha),\ 0\leq \beta \leq \pi,\ 0\leq \alpha,\,\gamma < 2\pi,
$$
 where 
$$ Z(\theta) = \left(\begin{array}{ccc} \cos \theta & -\sin \theta & 0 \\ \sin \theta & \cos \theta & 0 \\ 0 & 0 & 1 \end{array}\right), \quad \mbox X(\theta) = \left(\begin{array}{ccc} 1 & 0 & 0 \\ 0 & \cos \theta & -\sin \theta  \\ 0 & \sin \theta & \cos \theta \end{array}\right) 
$$
are  rotations about the $Z$- and $X$-axes, respectively. In the coordinates $\alpha, \beta, \gamma$, which are known as Euler angles, the Haar measure  of the group $SO(3)$ is given as 
$$
dg= \frac{1}{8\pi^2} \sin\beta d\alpha\,d\beta\,d\gamma. 
$$ 
In other words the following formula holds:
$$ \int_{SO(3)} f(g)\,dg = \int_0^{2\pi}\int_0^{\pi}\int_0^{2\pi} f(g(\alpha,\,\beta,\,\gamma))\frac{1}{8\pi^2} \sin\beta d\alpha\,d\beta\,d\gamma.
$$ 
Note that if $SO(2)$ is the group of rotations of $\mathbb{R}^{2}$ then the two-dimensional sphere $\mathbb{S}^{2}$ can be identified with the factor $SO(3)/SO(2)$.

We introduce Radon transform $\mathcal{R} f$ of a smooth function $f$ defined on $SO(3)$.\begin{definition}\label{groupRT}
If $\mathbb{S}^{2}$ is the standard unit sphere in $\mathbb{R}^{3}$ , then for a pair $(x,y)\in \mathbb{S}^{2}\times \mathbb{S}^{2}$ the value of the  Radon transform $\mathcal{R} f$ at $(x,y)$ is defined by the formula
\begin{equation}\label{RRR}
 (\mathcal{R} f)(x,y)   = \frac{1}{2\pi} \int_{\{g\in SO(3): x=gy\}} f(g) d\nu_g = 
 $$
 $$
 4\pi \int_{SO(3)} f(g)\delta_y(g^{-1}x) dg = (f*\delta_y)(x),\>\>\> (x,y)\in \mathbb{S}^{2}\times \mathbb{S}^{2},
\end{equation}
where $d\nu_g = 8\pi ^2 dg, $  and  $\delta_{y}$ is the measure concentrated on the set of all $g\in SO(3)$ such that $x=gy$. 
\end{definition}

\begin{remark}

Note, that  \textit{crystallographic X--ray transform} of a function $f$  is a  function on $\mathbb{S}^{2}\times \mathbb{S}^{2}$, which is defined by the following  formula
\begin{equation}\label{PDF}
 Pf(x,y) = \frac{1}{2}(\mathcal{R} f(x,y)+\mathcal{R} f(-x,y)),\>\>\>(x,y)\in \mathbb{S}^{2}\times \mathbb{S}^{2}.
 \end{equation}

\end{remark}

An orthonormal system in $L_2(S^2)$ is provided by the spherical harmonics $\{Y_k^i,\,k\in \mathbb{N}_0,\ i=1,\ldots , 2k+1\}.$ The subspaces  $\mathcal{H}_k:={\rm span}\,\{Y_k^i, i=1,\,\ldots ,\,2k+1\}$ spanned by the spherical harmonics of degree $k$ are the invariant subspaces of the quasi-regular representation 
$ T(g):\,f(x)\mapsto f(g^{-1}\cdot x), $
(where $\cdot $ denotes the canonical action of $SO(3)$ on $S^2$). Representation $T$ decomposes into $(2k+1)$-dimensional irreducible representations $\mathcal{T}_k$ in $\mathcal{H}_k, \>\>k\in  \mathbb{N}_0.$
The corresponding matrix coefficients are the Wigner-polynomials
$$ \mathcal{T}_k^{ij}(g) = \langle \mathcal{T}_k(g)Y_k^i, Y_k^j \rangle. $$
If $\Delta_{SO(3)}$ and $\Delta_{\mathbb{S}^{2}}$ are Laplace-Beltrami operators of invariant metrics on $SO(3)$ and $\mathbb{S}^{2}$ respectively, then 
\begin{equation}\label{eigenvalues}
 \Delta_{SO(3)}\mathcal{T}_k^{ij} = -k(k+1)\mathcal{T}_k^{ij}\quad \mbox{and}\quad \Delta_{S^2}Y_k^i = -k(k+1)Y_k^i. 
\end{equation}
 Using  the fact that $\Delta_{SO(3)}$ on the eigenspace $\mathcal{H}_k$ is just multiplication by $-k(k+1)$ we obtain 
$$ ||f||^2_{L_2(SO(3))} = 
||(4\pi)^{-1}(I-2\Delta_{\mathbb{S}^{2}\times \mathbb{S}^{2}})^{1/4}\mathcal{R} f||^2_{L_2(S^2\times S^2)} , $$
where $\Delta_{\mathbb{S}^{2}\times \mathbb{S}^{2}}=\Delta_1 + \Delta_2 $ is the Laplace-Beltrami  operator of the natural metric on $S^2\times S^2.$  Here $\Delta_{1}$ is acting on the first component and $\Delta_{2}$ is acting on the second component of the product  $\mathbb{S}^{2}\times \mathbb{S}^{2}$.

We define the following norm on the space  $ C^{\infty}(S^2\times S^2)$
$$
 ||| u |||^2 = ((I-2\Delta_{S^2\times S^2} )^{1/2}u,\,u)_{L_2(S^2\times S^2)}.
 $$
Because $\mathcal{R}$ is essentially an isometry between $L_{2}(SO(3))$ with the natural norm and $L_{2}(\mathbb{S}^{2}\times \mathbb{S}^{2})$ with the norm $|||\cdot|||$ the inverse of $\mathcal{R}$  is given by  its  adjoint operator. 
To calculate the adjoint operator we express the Radon transform $\mathcal{R}$ in another way. 
Going back to our problem in crystallography we first state that the great circle $C_{x,y}=\{g\in SO(3): g\cdot x = y\}$ in $SO(3)$ can also be described by the following formula
$$ C_{x,y}= x^{\prime}SO(2)(y^{\prime})^{-1}  := \{x^{\prime}h(y^{\prime})^{-1},\ h\in SO(2)\}, \quad x^{\prime},\,y^{\prime}\in SO(3), $$
where $x^{\prime}\cdot x_0 = x,\ y^{\prime}\cdot x_0 = y$ and $SO(2)$ is  the stabilizer of $x_0\in S^2.$ Hence,
\begin{eqnarray*}
\mathcal{R} f(x,y) = \int_{SO(2)} f(x^{\prime}h(y^{\prime})^{-1})\, dh = 4\pi \int_{C_{x,y}} f(g)\, dg  \\ =
 4\pi \int_{SO(3)} f(g)\delta_y(g^{-1}\cdot x)\, dg, \quad f\in L_2(SO(3)). 
 \end{eqnarray*}
 By using this representation one can find that the $L^2$-adjoint operator of $\mathcal{R}$ is given by
\begin{eqnarray}\label{adjoint_op}
  \left(\mathcal{R}^*u\right)(g) = (4\pi) \int_{S^2} (I-2\Delta_{S^2\times S^2})^{1/2}u(g\cdot
y,\,y)\,dy.
\end{eqnarray}

\begin{definition}[Sobolev spaces on $S^2\times S^2$] The Sobolev space 
  $H_{t}(S^2\times S^2),\,t\in \mathbb{R},$ is defined as the domain of the operator 
  $(I-2\Delta_{S^2\times S^2})^{\tfrac{t}{2}}$ with graph norm
		$$ ||f||_{H_{t}(S^2\times S^2)} = ||(I-2\Delta_{S^2\times S^2})^{\tfrac{t}{2}}f||_{L^2(S^2\times S^2)} ,
		$$
and 
the Sobolev space $H_t^{\Delta}(S^2 \times S^2),\,t\in \mathbb{R},$ is defined as the
 subspace of all functions $f\in H_t(S^2\times S^2)$ such $\Delta_1 f = \Delta_2 f.$ 
\end{definition}

\begin{definition}[Sobolev spaces on $SO(3)$] The Sobolev space $H_t(SO(3)),\,t \in \mathbb{R},$ is defined as the domain of the operator $(I-4\Delta_{SO(3)})^{\tfrac{t}{2}}$ with graph norm
			$$ |||f|||_t = ||(I-4\Delta_{SO(3)})^{\tfrac{t}{2}}f||_{L^2(SO(3))},\>\>f\in L_2(SO(3)). $$
\end{definition}

It is not difficult to prove the following theorems.
\begin{theorem}\label{mapRT}
    For any $t\geq 0$ the Radon transform on $SO(3)$ is an invertible  mapping
        \begin{align}
           \mathcal{R} : H_{t}(SO(3))\to H_{t+\frac{1}{2}}^{\Delta}(S^2\times S^2).
        \end{align}
     and
$$
  f(g) = \int_{S^2} (I-2\Delta_{S^2\times S^2})^{\tfrac{1}{2}}(\mathcal{R} f)(gy, y) dy = \frac{1}{4\pi} (\mathcal{R} ^* \mathcal{R} f )(g),\>\>\>g\in SO(3).
$$
 Thus, $\mathcal{R}^{-1}=\frac{1}{4\pi}\mathcal{R}^{*}. $
\end{theorem}
One can verify \cite{BEP} that the following relations hold
\begin{equation}\label{basis-action1}
   \mathcal{R}\mathcal{ T}_k^{ij}(x,y) = \mathcal{T}^{i1}_k(x) \overline{\mathcal{T}_k^{j1}(y)} =\frac{4\pi}{2k+1} Y_k^i(x)\overline{Y_k^j(y)},
\end{equation}
\begin{equation}\label{I1}
\Delta_{\mathbb{S}^{2}\times \mathbb{S}^{2}}\mathcal{R}f=2\mathcal{R}\Delta_{SO(3)}f,\>\>\>f\in H_{2}(SO(3)), 
\end{equation}
\begin{equation}\label{I2}
\left(1-2\Delta_{\mathbb{S}^{2}\times \mathbb{S}^{2}}\right)^{t/2}\mathcal{R}f=\mathcal{R}\left(1-4\Delta_{SO(3)}\right)^{t/2}f,\>\>\>f\in H_{t}(SO(3)), \>\>\>t\geq 0,
\end{equation}
\begin{equation}\label{I3}
\left(1-2\Delta_{\mathbb{S}^{2}\times \mathbb{S}^{2}}\right)^{t/2}g=\mathcal{R}\left(1-4\Delta_{SO(3)}\right)^{t/2}\mathcal{R}^{-1}g,
\end{equation}
where $g\in H_{t+1/2}^{\Delta}(\mathbb{S}^{2}\times \mathbb{S}^{2}), \>\>\>t\geq 0.$
\begin{theorem}[Reconstruction formula]\label{Recon}
Let
\begin{align*}
    G(x,y) = \mathcal{R} f(x,y) &= \sum_{k=0}^\infty \sum_{i,j=1}^{2k+1} \widehat G(k)_{ij} Y_k^i(x) 
    \overline{Y_k^j(y)} \in H_{\frac{1}{2}+t}^{\Delta}(S^2\times S^2),\
    t\geq 0,
\end{align*}
be a result of the Radon transform. Then the pre-image $f\in H_t(SO(3)),\ t\geq 0,$ is given by
\begin{align*}
    f &= \sum_{k=0}^\infty \sum_{i,j=1}^{2k+1} \frac{(2k+1)}{4\pi} \widehat G(k)_{ij} 
    \mathcal{T}^{ij}_k 
    \end{align*}
\end{theorem}

\subsection{Approximate inversion of the Radon transform on $SO(3)$ using  generalized splines}\label{approxSO}

Let $\{(x_{1}, y_{1}),...,(x_{N}, y_{N})\}$ be a set of pairs of points from $SO(3)$. In what follows we have to assume that our {\bf Independence Assumption} (\ref{independence}) holds. It takes now the following form: there are smooth functions $\phi_{1},...,\phi_{N}$ on $SO(3)$  with
$$
\mathcal{R} \phi_{\mu}(x_{\nu}, y_{\nu})=\delta_{\nu\mu}.
$$
But it is obvious that for this condition to satisfy it is enough to assume that submanifolds $\mathcal{M}_\nu=x_\nu SO(2) y_\nu^{-1}\subset SO(3)$ are pairwise different (not necessarily disjoint).

Let $f$ be  a function in $H_{t}(SO(3)),\>\>\>$ $t> \frac{1}{2}\left(\dim \> SO(3)\right)=3/2$ and
\begin{equation}\label{constraint}
v_{\nu}=\int_{\mathcal{M}_{\nu}}fdx,\>\>\>\nu=1,...,N,\>\>\>v=\{v_{\nu}\}.
\end{equation}
According to Definition \ref{interpolationdef} we use notation $s_{t}(f)=s_{t}(v)$ for a function in $ H_{t}(SO(3))$  such that for $\mathcal{M}_{\nu}=x_{\nu}  SO(2) y_{\nu}^{-1}$  it satisfies (\ref{constraint}) and minimizes the functional  
\begin{equation}\label{functional}
 u\rightarrow \|(1- 4\Delta_{SO(3)} )
^{t/2}u\|_{L_{2}(SO(3))}.
\end{equation}

In this situation the results of section \ref{splines} can be summarized in the following statement which was proved in \cite{BEP}.

\begin{theorem}
Let $\{(x_{1}, y_{1}),...,(x_{N}, y_{N})\}$ be a subset of  $SO(3)\times SO(3)$ such  that submanifolds $\mathcal{M}_\nu=x_\nu SO(2) y_\nu^{-1}\subset  SO(3),\>\>\nu=1,...,N,$ are pairwise different.

For a function $f$ in $H_{t}(SO(3)),\>\>\>$ $t>3/2,$ and a vector of measurements $v=\left(v_{\nu}\right)_{1}^{N}$ in (\ref{constraint}) 
 the solution of a constrained variational problem (\ref{constraint})-(\ref{functional}) is given by 
\begin{equation}
\label{splineSO}
s_{t}(f)=\sum_{k=0}^{\infty}\sum_{i,j=1}^{2k+1}c_{ij}^{k}(s_{t}(f))\mathcal{T}^{ij}_{k}=\sum_{k=0}^{\infty}trace\left(c_{k}(s_{t}(f))\mathcal{T}_{k}\right),
\end{equation}
where $\mathcal{T}^{ij}_{k}$ are the Wigner polynomials. The Fourier coefficients $c_{k}(s_{t}(f))$ of the solution are given by their matrix entries 
\begin{equation}
c_{ij}^{k}(s_{t}(f))=\frac{4\pi}{(2k+1)(1+k(k+1))^{t}}\sum_{\nu=1}^{N}\alpha_{\nu}(s_{t}(f) Y_k^i(x_{\nu})\overline{Y_k^j(y_{\nu})},
\end{equation}
where $\alpha(s_{t}(f))=\left(\alpha_{\nu}(s_{t}(f))\right)_{1}^{N}\in \mathbb{R}^{N}$ is the solution of 
\begin{equation}
\beta\alpha(s_{t}(f))=f,
\end{equation} 
with $\beta\in \mathbb{R}^{N\times N}$ given by 

\begin{align}
    \beta_{\nu\mu}= \sum_{k=0}^\infty (1+k(k+1))^{-t} C_k^{\frac12}(x_\nu\cdot y_\nu) C_k^{\frac12}(x_\mu\cdot y_\mu),
\end{align}
where  the Gegenbauer polynomials $C_k^{\frac12}$ are given by the formulas ( see  \cite{vilenkin1}, \cite{Vilenkin})
\begin{equation}
 \mathcal C_k^{\frac12} (x\cdot y)
  = \frac{4\pi}{2k+1} \sum_{i=1}^{2k+1} Y^i_k(x)\overline{Y^i_k(y)}
\end{equation}
 for all $x,y\in S^2$ and $k=0, 1, 2...\>\>$.
The function $s_{t}(f)\in H_{t}(SO(3))$ has the following properties:

\begin{enumerate}

\item $s_{t}(f)$ has the prescribed set of measurements $v\!=\left(v_{\nu}\right)_{1}^{N}$ at points $((x_{\nu}, y_{\nu}))_{1}^{N}$;

\item it minimizes the functional (\ref{functional});

\item  the solution (\ref{splineSO}) is optimal in the sense that for every sufficiently large $K>0$ it is the symmetry center of the convex bounded closed set of all functions $h$  in $H_{t}(SO(3))$  with $|||h|||_{H_{t}(SO(3))}\leq K$ which have the same set of measurements $v=\left(v_{\nu}\right)_{1}^{N}$ at points $((x_{\nu}, y_{\nu}))_{1}^{N}$.

\end{enumerate}

\end{theorem}

\subsection{ A sampling theorem for Radon transform of bandlimited functions on $SO(3)$}\label{SamplRSO}

According to  Theorem \ref{mapRT} if $f\in H_{t}(SO(3))$ then $\mathcal{R} f\in H_{t+1/2}^{\Delta}(\mathbb{S}^{2}\times \mathbb{S}^{2})$. 
Integral of $f$ over the manifold $x_{\nu}SO(2) y_{\nu}^{-1}$ is the value of $\mathcal{R} f$ at $(x_{\nu}, y_{\nu})$ where $\{(x_{\nu}, y_{\nu})\}=M_{\rho}\subset (\mathbb{S}^{2}\times \mathbb{S}^{2})$ is a $\rho$-lattice.  Note that dimension of the manifold $SO(3)$ is three and dimension  of the manifold $\mathbb{S}^{2}\times \mathbb{S}^{2}$ is four. 

For a fixed natural $l\geq 0$ and $t>3/2$ we  apply Theorem \ref{MainTheorem} to the manifold $\mathbb{S}^{2}\times \mathbb{S}^{2}$ and the set of distributions $F=\{\delta_{\nu}\}$ where
$$
\delta_{\nu}(g)=g(x_{\nu}, y_{\nu}),\>\>\>g\in C^{0}(\mathbb{S}^{2}\times \mathbb{S}^{2}),\>\>(x_{\nu},y_{\nu})\in M_{\rho}\subset \mathbb{S}^{2}\times \mathbb{S}^{2},
$$
to construct  spline  $s_{2^{l+2}+(t+1)}(\mathcal{R} f),$  which interpolates $\mathcal{R} f$ on $M_{\rho}=\{(x_{\nu}, y_{\nu})\}$ and minimizes functional \begin{equation}\label{functional2}
u\rightarrow \|(I- 2\Delta_{\mathbb{S}^{2}\times \mathbb{S}^{2}} )^{2^{l+1}+(t+1)/2}u\|_{L_{2}(  \mathbb{S}^{2}\times \mathbb{S}^{2}  )}.
\end{equation}
 Fourier coefficients of $s_{2^{l+2}+(t+1)}(\mathcal{R} f)$ with respect to the basis $\left\{Y_{k_{1}}^{i}Y_{k_{2}}^{j}\right\}$ can be obtained by using Theorem \ref{MainTheorem}. 
Let $s^{\Delta}_{2^{l+2}+(t+1)}(\mathcal{R} f)$ be the orthogonal projection (in the norm of $H_{t}(\mathbb{S}^{2}\times \mathbb{S}^{2})$) of $s_{2^{l+2}+(t+1)}(\mathcal{R} f)$ onto subspace 
$
H_{2^{l+2}+(t+1)}^{\Delta}(\mathbb{S}^{2}\times \mathbb{S}^{2}).
$
It means that $s^{\Delta}_{2^{l+2}+(t+1)}(\mathcal{R} f)$ has a representation of the form
\begin{equation}\label{shat}
s^{\Delta}_{2^{l+2}+(t+1)}(\mathcal{R} f)(x,y)=\sum_{k}\sum_{ij} c_{ij}^{k}(\mathcal{R} f)Y_k^i(x)\overline{Y_k^j(y)},\>\>\>(x,y)\in \mathbb{S}^{2}\times \mathbb{S}^{2},
\end{equation}
where
$$
c_{ij}^{k}(\mathcal{R} f)=\int_{ \mathbb{S}^{2}\times \mathbb{S}^{2}}s_{2^{l+2}+(t+1)}(\mathcal{R} f)(x,y)\overline{Y_k^i(x)}Y_k^j(y)dx dy
$$
are the Fourier coefficients of $s_{2^{l+2}+(t+1)}(\mathcal{R} f)$. 
Applying (\ref{basis-action1})  we obtain that the following function defined on $SO(3)$ 
\begin{equation}\label{bigS}
S_{2^{l+2}+(t+1)}(f)(x)=\mathcal{R}^{-1}\left(s^{\Delta}_{2^{l+2}+(t+1)}(\mathcal{R} f)\right)(x),
\end{equation}
has a representation 
$$
S_{2^{l+2}+(t+1)}(f)(x)=\sum_{k}\sum_{ij} \frac{2k+1}{4\pi}c_{ij}^{k}(\mathcal{R} f)\mathcal{T}_{k}^{ij}(x).
$$
Let us stress that these functions do not interpolate $f$ in any sense. However, the following approximation results were proved in \cite{BEP}.

\begin{theorem}\label{last-thm}
If $t>3/2$ then there exist a constant $C=C(t)>0$ such that for  any $\rho$-lattice  $M_{\rho}=\{(x_{\nu}, y_{\nu})\}\subset \mathbb{S}^{2}\times \mathbb{S}^{2}$ with sufficiently small $\rho>0$  and any sufficiently  smooth
function $f\in H_{t}(SO(3))$ the following inequality holds true
$$
\|\left(S_{\tau}(f)-f\right)\|_{H_{t}(SO(3))}\leq
C_{1}(l)\rho^{ 2^{l+2}} \|\mathcal{R} f\|_{H_{\tau}(\mathbb{S}^{2}\times \mathbb{S}^{2})},
$$
for $\tau=2^{l+2}+(t+1)$ and any $l=0, 1, ... \>\>.\>\>\>$ In particular, if a natural $k$ satisfies the inequality $t>k+3/2$, then 
\begin{equation}
\|S_{\tau}( f)-  f\|_{C^{k}(SO(3))}\leq 
C_{1}(l)\rho^{ 2^{l+2}}  \|\mathcal{R} f\|_{H_{\tau}(\mathbb{S}^{2}\times \mathbb{S}^{2})}
\end{equation}
for any $m=0, 1, ... .$ 
\end{theorem}

For an $\omega>0$ let us consider the space ${\bf E}_{\omega}(SO(3))$ of $\omega$-bandlimited functions on $SO(3)$ i.e. the span of all Wigner functions $T^{ij}_{k}$ with $k(k+1)\leq \omega$. As the formulas (\ref{eigenvalues}) and (\ref{basis-action1}) show the Radon transform of such function is $\omega$-bandlimited on $\mathbb{S}^{2}\times \mathbb{S}^{2}$ in the sense its Fourier expansion involves only functions $Y^{i}_{k}\overline{Y^{j}_{k}}$ which are eigenfunctions of $\Delta_{\mathbb{S}^{2}\times \mathbb{S}^{2}}$ with eigenvalue $-2k(k+1)\geq -2\omega$. Let ${\bf \mathcal{E}}_{\omega}(\mathbb{S}^{2}\times \mathbb{S}^{2})$ be the span of $Y_k^i(\xi)\overline{Y_k^j(\eta)}$ with $k(k+1)\leq  \omega$. Thus
 \begin{equation}\label{band-band}
  \mathcal{R}: {\bf E}_{\omega}(SO(3))\rightarrow{\bf  \mathcal{E}}_{\omega}(\mathbb{S}^{2}\times \mathbb{S}^{2}).
\end{equation}
For $f\in {\bf E}_{\omega}(SO(3))$ the following Bernstein-type inequality holds $$
\|(1-2\Delta_{\mathbb{S}^{2}\times \mathbb{S}^{2}})^{\tau}\mathcal{R} f\|_{L^{2}(\mathbb{S}^{2}\times \mathbb{S}^{2})}  \leq (1+4\omega)^{\tau}\|\mathcal{R} f\|_{L^{2}(\mathbb{S}^{2}\times \mathbb{S}^{2})}.
$$
For the proof of the next theorem we refer to \cite{BEP}.
\begin{theorem}(Sampling Theorem For Radon Transform).
\label{SamplingT}
If $t>3/2$ then there exist a constant $C=C(t)>0$ such that for  any $\rho$-lattice  $M_{\rho}=\{(x_{\nu}, y_{\nu})\}\subset \mathbb{S}^{2}\times \mathbb{S}^{2}$ with sufficiently small $\rho>0$  and any  $f\in {\bf E}_{\omega}(SO(3))$
  one has the estimate
\begin{equation}
\|\left(S_{\tau}(f)-f\right)\|_{H_{t}(SO(3))}\leq
$$
$$
2(1+4\omega)^{(t+1)/2}\left(C\rho^{2}(1+4\omega)\right)^{2^{l+1}}\|\mathcal{R}f\|_{L_{2}(\mathbb{S}^{2}\times \mathbb{S}^{2})},
 \end{equation}
for $\tau=2^{l+2}+(t+1)$ and any $l=0, 1, ... \>\>.\>\>\>$ In particular, if a natural $k$ satisfies the inequality $t>k+3/2$, then 
$$
\|S_{\tau}( f)-  f\|_{C^{k}(SO(3))}\leq
2(1+4\omega)^{(t+1)/2}\left(C\rho^{2}(1+4\omega)\right)^{2^{l+1}}\|\mathcal{R}f\|_{L_{2}(\mathbb{S}^{2}\times \mathbb{S}^{2})},
$$
for any $l=0, 1, ... .$ 
\end{theorem}
This Theorem \ref{SamplingT} shows that if 
$
\rho<\left(C(1+\omega)\right)^{-1/2}
$
then every $f\in {\bf E}_{\omega}(SO(3))$ is completely determined by a finite  set of values 
$$
\mathcal{R}f(x_{\nu}, y_{\nu})=\int_{\mathcal{M}_{\nu}} f,
$$
where $\mathcal{M}_{\nu}=x_{\nu}SO(2) y_{\nu}^{-1}\subset SO(3),\>\>\>\{(x_{\nu}, y_{\nu})\}=M_{\rho}\subset \mathbb{S}^{2}\times \mathbb{S}^{2}.$
Moreover, it shows that $f$  can be reconstructed as a limit (when $l$ goes to infinity) of functions $S_{2^{l+2}+t}( f)$ which were  constructed by using only the values of  the Radon transform $\mathcal{R} f(x_{\nu}, y_{\nu})$.

\section{Bandlimited and localized Parseval frames on compact homogeneous manifolds with applications to integral geometry}\label{b-l-p-f}

We now turn to a special class of the so-called  compact homogeneous manifolds
 \cite{H3}, \cite{H}, \cite{vilenkin1}, \cite{Vilenkin}, \cite{Z}. These manifolds (which have many symmetries) appear in applications very often. For compact homogeneous manifolds we are able to construct Parseval frames which share many properties with orthonormal bases.

\subsection { Compact homogeneous manifolds}\label{homman}
A {\it homogeneous compact manifold} $M$ is a
$C^{\infty}$-compact manifold  on which a compact
Lie group $G$ acts transitively. In this case $M$ is necessary of the form $G/K$,
where $K$ is a closed subgroup of $G$. The notation $L_{2}(M),$ is used for the usual Hilbert spaces,  with  invariant measure  $dx$ on $M$.

The Lie algebra $\textbf{g}$ of a compact Lie group $G$ 
is then a direct sum
$\textbf{g}=\textbf{a}+[\textbf{g},\textbf{g}]$, where
$\textbf{a}$ is the center of $\textbf{g}$, and
$[\textbf{g},\textbf{g}]$ is a semi-simple algebra. Let $Q$ be a
positive-definite quadratic form on $\textbf{g}$ which, on
$[\textbf{g},\textbf{g}]$, is opposite to the Killing form. Let
$X_{1},...,X_{n}$ be a basis of
$\textbf{g}$, which is orthonormal with respect to $Q$.
The operator
\begin{equation}\label{Casimir}
-X_{1}^{2}-X_{2}^{2}-\    ... -X_{d}^{2},    \ d=\dim\ G,
\end{equation}
is a bi-invariant operator on $G$, which is known as the Casimir operator (\cite{H3}, \cite{Z}).

Every element $X$ of the  Lie algebra of $G$ generates a vector
field on $M$, which we will denote by the same letter $X$. Namely,
for a smooth function $f$ on $M$ one has
$$
 Xf(x)=\lim_{t\rightarrow 0}\frac{f(\exp tX \cdot x)-f(x)}{t}
 $$
for every $x\in M$. In the future we will consider on $M$ only
such vector fields. The translations along integral curves of such
vector fields $X$ on $M$  can be identified with a one-parameter
group of diffeomorphisms of $M$, which is usually denoted as $\exp
tX, -\infty<t<\infty$. At the same time, the one-parameter group
$\exp tX, \>\>\>-\infty<t<\infty,$ can be treated as a strongly
continuous one-parameter group of operators acting on the space $L_{2}(M)$.  These operators act on functions according to the
formula
$$
f\rightarrow f(\exp tX\cdot x), \        t\in \mathbb{R}, \
  f\in L_{2}(M),\        x\in M.
$$
 The
generator of this one-parameter group will be denoted by $D_{X}$,
and the group itself will be denoted by
$$
e^{tD_{X}}f(x)=f(\exp tX\cdot x),\       t\in \mathbb{R}, \
     f\in L_{2}(M), \       x\in M.
$$

According to the general theory of one-parameter groups in Banach
spaces \cite{BB}, Ch.\ I,  the operator $D_{X}$ is a closed
operator in  $L_{2}(M)$.  

Since the operator (\ref{Casimir}) is  bi-invariant 
the corresponding operator on $L_{2}(M)$,
\begin{equation}\label{Casimir1}
\mathcal{L}=-D_{1}^{2}- D_{2}^{2}- ...- D_{d}^{2}, \>\>\>
       D_{j}=D_{X_{j}}, \  1\leq j\leq d, \        d=\dim \ G,
\end{equation}
commutes with all operators $D_{j}=D_{X_{j}}$. Here $D_{j}^{2}f=D_{j}\left(D_{j}f\right)$.
The operator $\mathcal{L}$, which is often called the Laplace operator, is
the image of the Casimir operator under differential of quazi-regular representation in $L_{2}(M)$.  The operator $\mathcal{L}$ is not necessarily the  Laplace-Beltrami operator  of the natural  invariant metric on $M$. But it coincides with such operator at least in the following cases:
1) if $M$ is an $n$-dimensional torus, 2) if the manifold $M$ is itself a compact  semi-simple Lie group  $G$ (\cite{H3}, Ch. II), 3) if $M=G/K$ is a compact symmetric space of
  rank one (\cite{H3},
Ch. II, Theorem 4.11).

The following theorem holds for any compact manifold  \cite{gpes-1}, \cite{pg}.

\begin{theorem}\label{cubature} (Cubature  formula with positive weights)
There exists  a  positive constant $c=c(M)$,    such  that if  $\rho=c\omega^{-1/2}$, then
for any $\rho$-lattice $M_{\rho}=\{x_{\nu}\}$, there exist strictly positive coefficients $\mu_{\nu}>0$ for which the following equality holds for all functions in ${\bf E}_{\omega}(M)$:
\begin{equation}
\label{cubway}
\int_{M}fdx=\sum_{x_{\nu}\in M_{\rho}}\mu_{\nu}f(x_{\nu}).
\end{equation}
Moreover, there exists constants  $\  c_{1}, \  c_{2}, $  such that  the following inequalities hold:
\begin{equation}
c_{1}\omega^{-n/2}\leq \mu_{\nu}\leq c_{2}\omega^{-n/2}, \ n=dim\ M.
\end{equation}
\end{theorem}

The following important result was obtained in \cite{gpes-1}, \cite{pg}.

\begin{theorem}(Product property)
\label{prodthm}
If $M=G/K$ is a compact homogeneous manifold and $\mathcal{L}$
 is the same as above, then for any $f$ and $g$ belonging
to ${\bf E}_{\omega}(\mathcal{L})$,  their product $fg$ belongs to
${\bf E}_{4d\omega}(\mathcal{L})$, where $d$ is the dimension of the
group $G$.

\end{theorem}

\subsection{Bandlimited and localized Parseval frames on homogeneous manifolds}\label{frames}

In this section  we assume that a manifold $M$ is homogeneous in the sense that it is of the form  $M=G/K,$  where $G$ is a compact Lie group and $K$ is its closed subgroup (see subsection \ref{homman}).  In this situation we consider  spaces of bandlimited functions ${\bf E}_{\omega}(\mathcal{L}),\>\>\>\omega>0,$ with respect to  the Casimir operator $\mathcal{L}$ that was defined in (\ref{Casimir1}).
Our goal is to construct a  tight bandlimited and localized frame in the space $L_{2}(M)$.

 Let $g\in C^{\infty}(\mathbb{R}_{+})$ be a monotonic function such that $supp\>g\subset [0,\>  2^{2}], $ and $g(s)=1$ for $s\in [0,\>1], \>0\leq g(s)\leq 1, \>s>0.$ Setting  $G(s)=g(s)-g(2^{2}s)$ implies that $0\leq G(s)\leq 1, \>\>s\in supp\  G\subset [2^{-2},\>2^{2}].$  Clearly, $supp\ G(2^{-2j}s)\subset [2^{2j-2}, 2^{2j+2}],\>j\geq 1.$ For the functions
 $
 \Phi(s)=\sqrt{g(s)}, \>\>\Phi(2^{-2j}s)=\sqrt{G(2^{-2j}s)},\>\>j\geq 1, \>\>\>
 $
 one has 
 $$
 \sum_{j\geq 0}\left(\Phi(2^{-2j}s)\right)^{2}=1, \>\>s\geq 0.
 $$
 Using the spectral theorem for $\mathcal{L}$ one  obtains
$$
\sum_{j\geq 0} \Phi^2(2^{-2j}\mathcal{L})f = f,\>\>f \in L_{2}(M),
$$
and taking inner product with $f$ gives
\begin{equation}
\label{norm equality-0}
\|f\|_{L_{2}(M)}^{2}=\sum_{j\geq 0}\left< \Phi^2(2^{-2j}\mathcal{L})f,f\right>=\sum_{j\geq 0}\|\Phi(2^{-2j}\mathcal{L})f\|_{L_{2}(M)}^{2} .
\end{equation}
Moreover, since the function $  \Phi(2^{-2j}s)$ has its support in  $
[2^{2j-2},\>\>2^{2j+2}]$ the elements $ \Phi(2^{-2j}\mathcal{L})f $ are bandlimited to  $
[2^{2j-2},\>\>2^{2j+2}]$.

Expanding $f \in L_{2}(M)$ in terms of eigenfunctions of $\mathcal{L}$ we obtain
$$
\Phi({2^{-2j} \mathcal{L}})f= \sum_i \Phi(2^{-2j}\lambda_i)c_i(f) u_i,\>\>\>c_i(f)=\left<f, u_{i}\right>.
$$  
Since for every $j$ function $\Phi(2^{-2j}s)$ is supported in the interval $[2^{2j-2}, 2^{2j+2}]$ the function $\Phi({2^{-2j}\mathcal{L}})f(x),\>\>x\in M, $ is bandlimited and belongs to ${\bf E}_{2^{2j+2}}({\mathcal{L}})$.
But then the function 
$\overline{\Phi({2^{-2j} \mathcal{L}})f(x)}$ is also in ${\bf E}_{2^{2j+2}}({ \mathcal{L}})$.
Since 
$$
|\Phi({2^{-2j}\mathcal{L}})f(x)|^2=\left[\Phi({2^{-2j} \mathcal{L}})f(x)\right]\left[\overline{\Phi({2^{-2j} \mathcal{L}})f(x)}\right],
$$
one can use  Theorem \ref{prodthm} to conclude that  
$
|\Phi({2^{-2j} \mathcal{L}})f|^2\in 
{\bf E}_{4d2^{2j+2}}({ \mathcal{L}}),
$
where $d=\dim G,\>\>{M}=G/K$.

To summarize, we proved, that for every $f\in L_{2}(M)$ we have the following decomposition 
\begin{equation}
\label{addto1sc}
\sum_{j\geq 0} \|\Phi({2^{-2j} \mathcal{L}})f\|^{2}_{L_{2}(M)} = \|f\|^{2}_{L_{2}(M)},\>\>\>\>\>
\left|\Phi({2^{-2j}\mathcal{L}})f(x)\right|^2\in 
{\bf E}_{4d2^{2j+2}}({ \mathcal{L}}).
\end{equation}
The next objective is to perform a discretization step. According to our Theorem \ref{cubature}  there exists a constant $c=c(M)>0$ such that for all integer $j$ if  
\begin{equation}
\label{rhoj}
\rho_j = c(4d2^{2j+2}+1)^{-1/2}\sim 2^{-j},\>\>\>d=\dim G, \>\>\>M=G/K, 
\end{equation}
then for any  $\rho_{j}$-lattice $M_{\rho_{j}}$ one can find coefficients $b_{j,k}$ with
\begin{equation}
\label{wtest1}
b_{j,k}\sim \rho_j^{n},\>\>\>n=\dim M,
\end{equation}
for which the following exact cubature formula holds
\begin{equation}
\label{cubl2}
\|\Phi({2^{-2j} \mathcal{L}})f\|^{2}_{L_{2}(M)} = \sum_{k=1}^{J_j}b_{j,k}\left|\left(\Phi({2^{-2j}\mathcal{L}})f\right)(x_{j,k})\right|^2,
\end{equation}
where $x_{j,k} \in M_{\rho_j}$, ($k = 1,\ldots,J_j = card\>(M_{\rho_j})$). 

For each $x_{j,k}$ we define the functions
\begin{equation}
\label{vphijkdf}
\psi_{j,k}(y) = \overline{\mathcal{K}^{\Phi}_{2^{-j}}}(x_{j,k},y) = \sum_i \overline{\Phi}(2^{-2j}\lambda_i) \overline{u}_i(x_{j,k}) u_i(y),
\end{equation}
\begin{equation}
\label{phijkdf}
\Psi_{j,k} = \sqrt{b_{j,k}} \psi_{j,k}.
\end{equation}
We find that for all $f \in 
L_2(M)$,
\begin{equation}
\label{parfrm}
\|f\|^{2}_{L_{2}(M)} = \sum_{j,k} |\langle f,\Psi_{j,k} \rangle|^2.
\end{equation}
Moreover, one can show  \cite{gpes-1}, \cite{pg}, \cite{Tay81} that the frame members $\psi_{j,k}$ are strongly localized on the manifold. All together it implies the following statement \cite{gpes-1}, \cite{pg}.
\begin{theorem}\label{main-frames}
If $M$ is a homogeneous manifold, then the set of functions $\{\Psi_{j,k}\},$ constructed in (\ref{vphijkdf})-(\ref{phijkdf}) has the following properties:

\begin{enumerate}

\item  $\{\Psi_{j,k}\},$ is a Parseval frame in the space $L_{2}(M)$. 

\item  Every function $\Psi_{j,k}$ is bandlimited to $[2^{2j-2}, 2^{2j+2}]$.

\item  For  any $N>0$ there exists a $C( N)>0$ such  that uniformly in $j$ and $k$ 
  \begin{equation}\label{localization}
    |\psi_{j,k}(x)|\leq 
    C( N) \frac{2^{nj}}{\left(1+\>2^{j}dist(x,\>x_{j,k})\right)^{N}},\>\>\>j\geq 0.
 \end{equation}

\item  The following reconstruction formula holds
\begin{equation}
\label{recon}
f = \sum_{j\geq 0}^{\infty}\sum_k \langle f,\Psi_{j,k} \rangle \Psi_{j,k} = 
\sum_{j\geq 0}^{\infty}\sum_k b_{j,k} \langle f,\psi_{j,k} \rangle \psi_{j,k},\>\>\>f \in L_2(M),
\end{equation}
with convergence in $L_2(M)$. 

\end{enumerate}

\end{theorem}

By using Theorems \ref{cubature} and \ref{prodthm} one can easily obtain a following {\bf exact}  discrete formula for Fourier coefficients which uses only  samples of $f$ on a sufficiently dense lattice.

\begin{theorem}\label{DFT}
If $M$ is a homogeneous compact manifold then there exists a $c=c(M)>0$ such that for any $\omega>0$, if 
$
\rho_{\omega}=c(\omega+1)^{-1/2},
$  
then
for any $\rho_{\omega}$-lattice $M_{\rho_{\omega}}=\{x_{\nu}\}_{\nu=1}^{N_{\omega}}$ of $M$, there exist positive weights
$
\mu_{\nu}\asymp (\omega+1)^{-n/2}, \>\>\>n=\dim M, 
 $
 such that for every  function $f$  in ${\bf E}_{\omega}(\mathcal{L})$ the  Fourier coefficients $c_{i}\left( f\right)$  
 $$
 c_{i}(f)=\int_{M}f\overline{u_{i}},\>\>\>\> - \mathcal{L}u_{i}=\lambda_{i}u_{i},\>\>\>\lambda_{i}\leq \omega,  
 $$
  are given by the formulas 
 \begin{equation}\label{DFTF}
c_{i}\left( f\right)=\sum_{\nu=1}^{N_{\omega}}\mu_{\nu}  f(x_{\nu})\overline{u_{i}}(x_{\nu}).
\end{equation}
\end{theorem}

Theorems \ref{main-frames}, \ref{DFT}, \ref{cubature}, and \ref{prodthm} can be used to prove another exact discrete formula for Fourier coefficients which involves frame functions. 

 \begin{theorem}\label{D-RepFrame}
 For a compact homogeneous manifold $M$
there exists a constant $c=c(M)>0$ such that for any  natural  $J$ if 
$$
\rho_{J}=c2^{-J}
$$
 then
for any $\rho_{J}$-lattice $M_{\rho_{J}}=\{x^{*}_{\nu}\}_{\nu=1}^{N_{\omega}},$ there exist positive weights 
$$
\mu_{\nu}^{*} \asymp2^{-Jn},\>\>\>n=\dim\>M,
$$ 
such that the following  formula holds
\begin{equation}\label{discrrep}
f = \sum_{j=0}^{J}\sum_{k=1}^{n_{j}} \sum_{\nu=1}^{N_{J}}\mu_{\nu}^{*}f(x^{*}_{\nu})\psi_{j,k}(x^{*}_{\nu})\psi_{j,k},\>\>\>f\in {\bf E}_{\omega}({\mathcal L}),\>\>\>\omega=2^{2J}-1.
\end{equation}
\end{theorem}

\subsection{ Exact formulas for Fourier coefficients  of  a bandlimited function $f$  on $\mathbb{S}^{n}$ from  a finite number of samples of  its Funk-Radon transform} \label{exactFCS}
 Theorem \ref{DFT} can be used to obtain a discrete inversion formula for $R$.

\begin{theorem}(Discrete Inversion Formula)\label{SamplingTh}
There exists a $c=c(M)>0$ such that for any $\omega>0$, if 
$
\rho_{\omega}=c\omega^{-1/2},
$  
then
for any $\rho_{\omega}$-lattice $M_{\rho_{\omega}}=\{x_{\nu}\}_{\nu=1}^{m_{\omega}}$ of $\mathbb{S}^{n}$, there exist positive weights
$
\mu_{\nu}\asymp \omega^{-d/2}, 
 $
 such that for every  function $f$  in ${\bf E}_{\omega}(\mathbb{S}^{n})$ the  Fourier coefficients $c^{i}_{k}\left(R f\right)$  of its  Radon transform, i.e.
 $$
Rf(x)=\sum_{i,k} c^{i}_{k}\left( R f\right) Y^{i}_{k}(x),\>\>\>\>\>\>\>\>\>k(k+1)\leq \omega,\>\>\>\>\>x\in \mathbb{S}^{n},
 $$
  are given by the formulas 
 \begin{equation}
c^{i}_{k}\left(R f\right)=\sum_{\nu=1}^{m_{\omega}}\mu_{\nu} \left(R f\right)(x_{\nu}) Y^{i}_{k}(x_{\nu}).
\end{equation}
The function $f$ can be reconstructed by means of  the formula
\begin{equation}\label{Exact555}
f=\frac{\sqrt{\pi}}{\Gamma((d+1)/2)}\sum_{k}\sum_{i}\frac{c^{i}_{k}\left(R f\right)}{r_{k}} Y_{k}^{i},
\end{equation}
in which $k$ runs over all natural even numbers  such that $k(k+1)\leq \omega$ and $r_{k}$ are defined in (\ref{r}).
\end{theorem}

\begin{proof}

Note that in the case when $M$ is the rank one compact symmetric space (in particular, the sphere $\mathbb{S}^{n}$ or any of the projective spaces $P^{d},\>\mathbb{C}P^{d},\>\mathbb{Q}P^{d}$) the constant $4d$ in Theorem \ref{prodthm}
can be improved. Namely, 
if $M=G/H$ is a compact symmetric space of rank one   then for any $f$ and $g$ belonging
to ${\mathbf E}_{\omega}(\mathcal{L})$,  their product $fg$ belongs to
${\mathbf E}_{2\omega}(\mathcal{L})$.

Using this fact 
we obtain  that if $k(k+1)\leq \omega$ then  every product 
$
 Y^{i}_{k}\overline{Y^{j}_{k}}$, where $k(k+1)\leq \omega$ belongs to   ${\bf E}_{2\omega}(\mathbb{S}^{n})$.

By  Theorem \ref{cubature} there exists  a  positive constant $c=c(M)$,    such  that if  $\rho_{\omega}=c\omega^{-1/2}$, then
for any $\rho_{\omega}$-lattice $M_{\rho_{\omega}}=\{x_{\nu}\}^{m_{\omega}}_{\nu=1}$  on  $\mathbb{S}^{n}$ 
there exist  a set of positive weights 
$
\mu_{\nu}\asymp \omega^{-d/2}
$
such that 
\begin{equation}
c^{i}_{k}\left( R f\right)=\int_{\mathbb{S}^{n}} \left(R f\right)(x) \overline{Y^{i}_{k}(x)} dx=
\sum_{\nu=1}^{m_{\omega}}\mu_{\nu} \left(R f\right)(x_{\nu}) \overline{Y^{i}_{k}(x_{\nu})}.
\end{equation}
Thus,
$$
\left(Rf\right)(x)=\sum_{k,i}c^{i}_{k}\left( R f\right) \overline{Y^{i}_{k}(x)}.
$$
Now the reconstruction formula of Theorem \ref{Recon} implies the formula (\ref{Exact555}).

\end{proof}

\bigskip

Without going to details we just mention that similar results can be obtained for the hemispherical Radon transform on $\mathbb{S}^{n}$ (see \cite{Pes04c}) and for the case of Radon transform of functions on domains see \cite{Pes12}.

\subsection{ Exact formulas for Fourier coefficients  of  a bandlimited function $f$  on $SO(3)$ from  a finite number of samples of  $ \mathcal{R} f$}\label{Discrete}

Let $M_{\rho}=\{(x_{\nu}, y_{\nu})\}$ be  a metric
$\rho$-lattice of  $\mathbb{S}^{2}\times \mathbb{S}^{2}$.
In what follows $\mathcal{E}_{\omega}(\mathbb{S}^{2}\times \mathbb{S}^{2})$ will denote  the span in the space $L^{2}(\mathbb{S}^{2}\times \mathbb{S}^{2})$ of all $Y_k^i\overline{Y_k^j}$  with $k(k+1)\leq \omega$. Theorem \ref{DFT} implies the following exam discrete reconstruction formula which uses only  samples of $\mathcal{R} f$ on a sufficiently dense lattice.

\begin{theorem}(Discrete Inversion Formula \cite{BP})\label{SamplingTh}
There exists a $c=c(M)>0$ such that for any $\omega>0$, if 
$
\rho_{\omega}=c\omega^{-1/2},
$  
then
for any $\rho_{\omega}$-lattice $M_{\rho_{\omega}}=\{(x_{\nu}, y_{\nu})\}_{\nu=1}^{m_{\omega}}$ of $\mathbb{S}^{2}\times \mathbb{S}^{2}$ , there exist positive weights
$
\mu_{\nu}\asymp \omega^{-2}, 
 $
 such that for every  function $f$  in ${\bf E}_{\omega}(SO(3))$ the  Fourier coefficients $c_{i,j}^{k}\left( \mathcal{R} f\right)$  of its  Radon transform, i.e.
 $$
  \mathcal{R} f(x,y)=\sum_{i,j,k} c_{i,j}^{k}\left( \mathcal{R} f\right) Y^{i}_{k}(x)\overline{Y^{j}_{k}}(y),\>\>\>\>\>\>\>\>\>k(k+1)\leq \omega,\>\>\>\>\>(x,y)\in \mathbb{S}^{2}\times \mathbb{S}^{2},
 $$
  are given by the formulas 
 \begin{equation}
c_{i,j}^{k}\left(\mathcal{R} f\right)=\sum_{\nu=1}^{m_{\omega}}\mu_{\nu} \left(\mathcal{R} f\right)(x_{\nu},y_{\nu}) Y^{i}_{k}(x_{\nu})\overline{Y^{j}_{k}}(y_{\nu}).
\end{equation}
The function $f$ can be reconstructed by means of  the formula
\begin{equation}\label{Exact}
f(g)=\sum_{k}\sum_{i,j}^{2k+1}\frac{(2k+1)}{4\pi} c_{i,j}^{k}\left( \mathcal{R} f\right)\mathcal{T}_{k}^{i,j}(g),\>\>\>\>\>g\in SO(3),
\end{equation}
in which $k$ runs over all natural numbers  such that $k(k+1)\leq \omega$.
\end{theorem}

\end{document}